\documentclass[12pt]{article}
  
\usepackage{amsmath}
\usepackage{amsfonts}
\usepackage{amsthm}     
\usepackage{amssymb}  
\usepackage{latexsym}
\usepackage{euscript}

\newcommand{\msf}[1]{\mathsf {#1}}

\newcommand{\mbf}[1]{{\mathbf {#1}}}
\newcommand{\mcal}[1]{{\mathcal {#1}}} 

\newtheorem{theorem}{Theorem}   [section]
\newtheorem{lemma}[theorem]{Lemma}
\newtheorem{corollary}[theorem]{Corollary}
\newtheorem{proposition}[theorem]{Proposition}

\theoremstyle{definition}\newtheorem{remark}[theorem]{Remark}
\theoremstyle{definition}
\theoremstyle{definition}\newtheorem{example}[theorem]{Example}
\theoremstyle{definition}

\theoremstyle{definition}
\theoremstyle{definition}

\theoremstyle{plain}\newtheorem*{theorem*}{Theorem}
\theoremstyle{plain}\newtheorem*{corollary*}{Corollary}
\theoremstyle{remark}\newtheorem*{remark*}{Remark}
\theoremstyle{remark}\newtheorem*{remarks*}{Remarks}
\theoremstyle{definition}\newtheorem*{conjecture*}{Conjecture}
\theoremstyle{definition}\newtheorem*{definition*}{Definition}
\theoremstyle{definition}\newtheorem*{example*}{Example}
\theoremstyle{definition}\newtheorem*{question*}{Question}
\theoremstyle{definition}\newtheorem*{questions*}{Questions}
\theoremstyle{definition}\newtheorem*{hypothesis*}{Hypothesis}
\theoremstyle{definition}\newtheorem{model}[theorem]{Model}

\def\qed{\ifhmode\unskip\nobreak\fi\ifmmode\ifinner\else\hskip5pt\fi\fi
 \hfill\hbox{\hskip5pt\vrule width4pt height6pt depth1.5pt\hskip1pt}}
 
\newcommand{\ov}[1]{\ensuremath{\overline{#1}}}

\newcommand{\Int}{\ensuremath{\operatorname{\mathsf{Int}}}}

\newcommand{\RR}{\ensuremath{{\mathbb R}}}     
\newcommand{\R}[1]{\ensuremath{{\mathbb R}^{#1}}} 

\newcommand{\FF}{\ensuremath{{\mathbb F}}}     

\renewcommand{\H}[1]{\ensuremath{{\mathbb H}^{#1}}} 

\newcommand{\Rp}{\ensuremath{{\mathbb R}_+}}   
\newcommand{\CC}{\ensuremath{{\mathbb C}}}     

\newcommand{\ZZ}{\ensuremath{{\mathbb Z}}}	   

\newcommand{\om}[1]{\ensuremath{\omega(#1)}}
\newcommand{\emp}{\ensuremath{\emptyset}}

\newcommand {\pd}[1] {\ensuremath{\frac{\partial}{\partial #1}}}

\def\d#1dt{\frac{d#1}{dt}}    

\newcommand{\eps}{\ensuremath{\epsilon}}

\newcommand{\Lam}{\ensuremath{\Lambda}}

\newcommand{\Fix}[1]{\ensuremath{\operatorname{\mathsf {Fix}}(#1)}}
\newcommand{\co}{\colon\thinspace} 

\renewcommand{\ll}{\ensuremath{\mathfrak l}} 

\renewcommand{\ss}{\ensuremath{\mathfrak {s}}}

\renewcommand{\ll}{\ensuremath{\mathfrak l}} 

\newcommand{\sm}[1]{\ensuremath{\setminus #1}}

\def \za{\alpha}
\def \zb{\beta}
\def \zg{\gamma}

\def \zl{\lambda}
\def \zm{\mu}

\def \zp{\pi}

\def \zr{\rho}

\def \zf{\varphi}

\def \zq{\psi}
\def \zw{\omega}

\def \zF{\Phi}

\def \zsu{\sum}

\def \zin{\cap} 
\def \zing{\bigcap} 
\def \zun{\cup}  
  
\def \zex{\wedge}

\def \zdi{\oplus} 
  
\def \zte{\otimes}  

\def \zmm{\pm}
 
\def \zpu{\cdot}  
\def \zpor{\times}
\def \zci{\circ}  

\def \zas{\ast}

\def \zmei{\leq}
\def \zmai{\geq}
\def \zco{\subset}

\def \zcco{\supset}

\def \zpe{\in}

\def \zeq{\equiv}

\def \znoi{\neq}

\def \zpar{\partial}
\def \zinf{\infty}
\def \zva{\emptyset}

\def \zfl{\rightarrow}

\def \zsii{\Leftrightarrow}

\def \zbv{\mid}

\def \z/{\over}
  
\def\emp{\varnothing}

\def\mylabel#1{\label{#1}}

\newcommand{\V}{\ensuremath{{\mcal V}}}
	
\newcommand{\Z}[1]{\ensuremath{{\msf  Z ( #1)}}}
\renewcommand{\ss}  {\ensuremath{{\mathfrak {s}}}}
\renewcommand{\dim} {\ensuremath{{\mathsf {dim}}}}
\renewcommand{\om}{\omega}

\newcommand {\A}{\mcal A}
\newcommand {\B}{\mcal B}
\newcommand {\G}{\mcal G}

\renewcommand {\H}{\mcal H}
\newcommand {\I}{\mcal I}

\def\emp{\varnothing}
\begin{document}

\title{\bf Zero sets of Lie algebras of analytic vector fields on real
  and complex 2-dimensional manifolds, II} 
  \author{{\bf Morris
    W. Hirsch}\thanks{I thank Paul Chernoff, David Eisenbud, Robert
    Gompf, Henry King, Robion Kirby, Frank Raymond, Helena Reis and
    Joel Robbin for helpful discussions.}  \\ Department of
  Mathematics\\ University of Wisconsin at Madison\\ University of
  California at Berkeley\\ {\bf F.-J. Turiel}\\ Department of Geometry
  and Topology\\ University of Malaga }
\maketitle

\begin{abstract} 

On a real ($\mathbb F=\mathbb R$) or complex ($\mathbb F=\mathbb C$)
analytic connected 2-manifold $M$ with empty boundary consider two
vector fields $X,Y$.   We say that $Y$ {\it tracks} $X$ if $[Y,X]=fX$
for some continuous function $f\co M\rightarrow\mathbb F$. Let $K$
be a compact subset of the zero set ${\msf Z}(X)$ such that ${\msf
  Z}(X)-K$ is closed, with nonzero Poincar\'e-Hopf index (for example
$K={\msf Z}(X)$ when $M$ is compact and $\chi(M)\neq 0$) and let
$\mathcal G$ be a finite-dimensional Lie algebra of analytic vector
fields on $M$. 
\smallskip

{\bf Theorem.} Let $X$ be analytic and nontrivial. If every element of
$\mathcal G$ tracks $X$ and, in the complex case when ${\msf i}_K (X)$
is positive and even no quotient of $\mathcal G$ is isomorphic to
${\mathfrak {s}}{\mathfrak {l}} (2,\mathbb C)$, then $\mathcal G$ has
some zero in $K$.
\smallskip

{\bf Corollary.} If $Y$ tracks a nontrivial vector field $X$, both of
them analytic, then $Y$ vanishes somewhere in $K$.
\smallskip

Besides fixed point theorems for certain types of transformation
groups are proved. Several illustrative examples are given.

\end{abstract}

\tableofcontents
\section{Introduction}   \mylabel{seM1}

A fundamental issue in Dynamical Systems is deciding whether a vector
field has a zero. When the manifold is compact with empty boundary and
nonvanishing Euler characteristic, a positive answer is given by the
celebrated {\sc Poincar\'e-Hopf} Theorem.

Determining whether two or more vector fields have a common zero is
 more challenging.  This problem is closely related to the question of
which general classes of noncompact transformation groups must have
fixed points.  Early results along these lines include {\sc A. Borel}
\cite{Borel56}, {\sc E. Lima} \cite{Lima64}, {\sc A. Sommese}
\cite{Sommese73} and {\sc J. Plante} \cite {Plante86}.

Throughout this paper manifolds are real or complex with corresponding
ground field denoted by $\FF= \RR$ or $\CC$.  Each manifold has a
analytic structure (meaning holomorphic in the complex case), and
empty boundary unless the contrary is indicated. 
The closure of subset $\Lam$ of a topological space is denoted by $\ov
\Lam$ and the interior by $\Int \Lam$.

Let $N$ be a manifold.   
The dimension of $N$ over the ground field is denoted by $\dim_\FF N$, 
or by $\dim N$ when the ground field $\FF$ is clear from the context.
$\V (N)$ is the vector space over
$\FF$ of continuous vector fields on $N$, with the compact-open
topology, while $\mcal V^{k} (N)\zco\mcal \V (N)$, $k=1,\dots,\zinf$, 
(respectively $\mcal V^\om (N)$)denotes the subalgebra of vector fields that are 
$C^k$-differentiable (respectively analytic) . Of course
$\mcal V^1(N)=\mcal V^\om (N)$ when $N$ is complex.
  
Consider a set $\mathcal A$ of vector fields on a manifold $N$.  The
set of their common zeros is ${\msf Z}({\mcal A}):=\bigcap_{X\in\mcal
  A}{\msf Z}(X)$, where ${\msf Z}(X)$ is the set of zeros of $X$.  A
set $S\subset P$ is {\em $X$-invariant} if it contains the orbits
under $X$ of its points.  When this holds for all $X$ in $\mcal A$ one
will say that $S$ is {\em $\mcal A$-invariant}. 

Let $X$ be a continuous vector field on real manifold $N$ with empty
boundary. Given an open set
$U\zco N$ with compact closure $\overline U$ such that
$(\overline U\,\verb=\=\,U)\cap {\msf Z}(X)=\emp$, the {\it index} 
${\msf i}(X,U)$ of $X$ on $U$ is defined as the Poincar\'e-Hopf index of
 any sufficiently close approximation $X'\in \V (U)$ to $X|U$ (in the compact open
 topology) such that $\Z {X'}$ is finite.  Equivalently: ${\msf i}(X, U)$ is
 the intersection number of $X|U$ with the zero section of the tangent
 bundle ({\sc Bonatti} \cite{Bonatti92}).  This number is independent
 of the approximation, and is stable under perturbation of $X$ and
 replacement of $U$ by smaller open sets containing
 ${\msf Z}(X) \cap U$

When $X$ is $C^1$ and generates the local flow $\phi$ on $M$, for
sufficiently small $t >0$ the index ${\msf i}(X, U)$ equals the
fixed-point index $I (\phi_t|U)$ defined by {\sc Dold} \cite {Dold72}.

A {\em block} of zeros for $X$, or an {\em $X$-block}, is a compact
set $K\subset {\msf Z}(X)$ such that ${\msf Z}(X) \,\verb=\=\,K$ is
closed. This is equivalent to the existence of a precompact open neighborhood $U$
of $K$ such that $ {\msf Z}(X) \cap \ov U=K$.  We say that such a $U$ is 
{\it isolating} for  $(X,K)$. 

When $U$ is isolating for $(X, K)$ then ${\msf i}(X,U)$ does not
depend on the choice of $U$: it is completely determined by $X$ and
$K$.  The {\it index of $X$ at $K$} is defined to be ${\msf i}_K (X):={\msf
  i}(X,U)$. An $X$-block $K$ is {\em essential} provided ${\msf i}_K
(X)\ne 0$.  When this holds $K\ne\varnothing$.

Notice that the notions of ``block'' and ``index'' are well defined for
holomorphic vector fields on a complex manifold, since these are also
vector fields--- sections of the tangent bundle--- on the underlying
real manifold.
\begin{theorem*} [{\sc Poincar\'e-Hopf} {\rm \cite{Hopf25, Poincare85}}]        
If $N$ is compact, ${\msf i}_{{\msf Z}(X)}(X)={\msf i}(X, N) =\chi
(N)$ for every continuous vector fields $X$ on $N$.
 \end{theorem*}
\noindent
For calculations of the index in more general settings see {\sc Morse
  \cite{Morse29}, Pugh \cite{Pugh68}, Gottlieb \cite{Gottlieb86},
  Jubin \cite{Jubin09}}.

This paper was inspired by a remarkable result of C. Bonatti, which
does not require compactness of $N$:\footnote{
 ``The demonstration of this result involves a beautiful and quite
  difficult local study of the set of zeros of $X$, as an analytic
  $Y$-invariant set.''\ ---{\sc P.\ Molino} \cite{Molino93}}

\begin{theorem*} [{\sc  Bonatti} \cite{Bonatti92}]  
Assume $N$ is a real manifold of dimension $\le 4$ and 
$X, Y$ are
analytic vector fields on $N$ such that $[X, Y]= 0$.  Then ${\msf
  Z}(Y)$ meets every essential $X$-block.\footnote{
In \cite{Bonatti92} this is stated for $\dim\, (N) = 3$ or $4$.  If
$\dim\, (N) =2$ the same conclusion is obtained by applying the
3-dimensional case to the vector fields $X \times t\pd t, \ Y\times
t\pd {t}$\, on $N\times\RR$.}
\end{theorem*}

We say that $Y$ {\em tracks $X$} if $Y, X$ are $C^1$ vector fields on
a real or complex manifold $N$ and $[Y,X] =fX$ for some
continuous function $f\co N\rightarrow\FF$, which we will call
the {\it tracking function}. When this holds for all $Y$ in a set  $\A$ of 
vector fields we say that $\A$ tracks $X$.

For instance:  if $X$ is basis of a 1-dimensional ideal of a Lie
algebra $\G$ of vector fields then $\G$ tracks $X$; but the converse
does not always hold, even when $\G$ is finite dimensional (see Example
\ref{exM1}).

Notice that if $\FF=\CC$ and $X$ is nontrivial on each connected component of $N$, 
then $f$ is necessarily  holomorphic, but $f$ might be merely  continuous if $\FF=\RR$ 
(see Section \ref{seM3})
  
In the rest of this section  we  postulate:
{\em \begin{itemize}  
\item  $M$ is a real or complex 2-dimensional connected manifold 
with empty boundary,  

\item $X\in \V^\om (M)$ is nontrivial and $K$ is an essential $X$-block,

\item $\G\subset \V^\om (M)$ is a  Lie algebra that is
finite-dimensional over the ground field and tracks $X$.

\end{itemize}
}
This is our main result:
\begin{theorem}[\sc Main]     \mylabel{th:main}
Assume the following condition holds:
\begin{description}

\item[(*)] If  $M$ is  complex  and  ${\msf i}_K (X)$ is
even, no  quotient of $\G$ is isomorphic to $\ss\ll (2,\CC)$.
\end{description}
Then ${\msf Z}({\G})\cap K\ne\varnothing$.

\end{theorem}
\noindent The proof is given in Section \ref{seM5}.

\medskip

Hypothesis (*) cannot be deleted, as is shown by Example \ref{exM1} and Theorem
\ref{thM1A}. Nevertheless in the compact case the ``exceptions'' to Theorem 
\ref{th:main} are completely described by Theorem \ref{thM1A} and Remark \ref{reMla}.
Summarizing these two results we can state:

\begin{theorem}     \mylabel{thNew1}

Assume $M$ is compact and complex, $\chi (M)\znoi 0$ and  
${\msf Z}({\G})\zin{\msf Z}(X)=\zva$. Then $M$ is a holomorphic $\CC P^1$-bundle
over $\CC P^1$ and $\G$ is isomorphic to either $\ss\ll(2,\CC)$, 
${\mathfrak{g}}{\mathfrak {l}}(2,\CC)$ or the product of $\ss\ll(2,\CC)$ with
the affine algebra of $\CC$. 
\end{theorem}

Now consider the particular case in which $M$ is a compact complex surface 
endowed with a Lie subalgebra $\G\zco\V^\zw (M)$ and $\chi (M)\znoi 0$. As an application 
of Theorem \ref{thNew1} one shows that if $\G$ is isomorphic to 
${\mathfrak{g}}{\mathfrak {l}}(2,\CC)$ and ${\msf Z}(\G)\znoi\zva$, then 
there is a diffeomorphism (biholomorphism) between $M$ and $\CC P^2$ which
transforms $\G$ in the Lie subalgebra of $\V^\zw (\CC P^2 )$ consisting of those projective 
vector fields that on $\CC^2 \zco\CC P^2$ are linear (see Example \ref{exM7}).
Therefore any effective holomorphic action of $GL(2,\CC)$ on $M$ with  some fixed point is
equivalent to the natural action of $GL(2,\CC)$ on $\CC P^2$.

Let us return to the general case.

\begin{corollary}               \mylabel{coM4}
If  $\G\subset \V^\om (M)$ is  a solvable Lie algebra tracking $X$,
then ${\msf Z}({\G})\cap 
K\ne\varnothing$. 
\end{corollary}
\begin{proof} Follows from the main theorem because
  solvability of $\G$ validates hypothesis (*).
\end{proof}

\begin{theorem}         \mylabel{thM2}
Consider a compact,   complex 2-manifold $M$ with
$\chi(M)\znoi 0$.  If $\G\subset \V^\om (M)$ is a 
 solvable Lie algebra
then ${\msf Z}({\G})\znoi\emp$.
\end{theorem}

\begin{proof} $\G$ is isomorphic to a Lie algebra of upper triangular matrices by
Lie's Theorem \cite{Jacobson79}. If $\G\znoi\{0\}$, some $X\in\G$ spans a
1-dimensional ideal and is thus tracked by $\G$. Since ${\msf Z}(X)$
is an essential $X$-block by the Poincar\'e-Hopf Theorem, the
conclusion follows from Theorem \ref{th:main} applied to the
essential $X$-block $K:= \Z X$.
\end{proof}

\begin{remark} \mylabel{reMnr}

The analog of Theorem \ref{thM2} for real manifolds is not true: The
vector fields on $\R 2$,
\[
 \pd x, \qquad \pd y, \qquad  y\pd x + x \pd y,
\]
 extend over the real projective plane $\RR P^2$ to span a
 3-dimensional solvable Lie algebra $\G \subset \V^\om (\RR P^2)$, and
 $\Z \G =\emp$.  But the proof Theorem \ref{thM2} can be adapted to
 show that the real analog holds provided $\G$ is {\em
 supersolvable}--- faithfully represented by upper triangular real
 matrices.
\end{remark}

\subsection{Lie group actions}   \mylabel{sec:liegrp}


Let $G$ denote Lie group over the
same ground field $\FF$ as $M$.

 An {\em action} of $G$ on $M$ is an $\FF$-analytic map
$$\za\co G\zpor M\zfl M$$
 such that the map
\[g^\za\co p\to \za(g,p)
\]
 is a homomorphism from $G$ to the group of $\FF$-analytic
 diffeomorphisms of $M$ (see {\sc Palais \cite {Palais57}}).  This
 action is also denoted by $(\alpha, G, M)$.  Its {\em fixed point
   set} is
\[\Fix \alpha:=\{p\in M\co g^\alpha (p)=p, \quad (g\in G)\}.\]
The action is {\em effective} if its kernel is trivial, and {\em
 almost effective} if  its kernel is  discrete.

An analytic action $\za$ of $G$ on $M$ gives rise to a homomorphism
$\msf d_\za$ from the Lie
algebra $\G $ of $G$ onto a subalgebra $\G^\za  \subset \V^\om (M)$,
the {\it infinitesimal action} 
determined by $\alpha$.  Note that $\msf d_\za$ is injective
if and only if the action is almost effective.  When $G$ is
connected, $\Fix \alpha ={\msf Z}(\G^\za )$.

{\em \begin{itemize}  
\item In the next two results $G$ is a connected Lie group.
\end{itemize}}

\begin{theorem}        \mylabel{thM3} 
Assume:

\begin{description}

\item[(a)] $M$ is  compact and   $\chi (M)\ne 0$.

\item[(b)] $G$  contains a 1-dimensional
  normal subgroup.

\item[(c)]  If $M$ is complex and $\chi (M)$ is positive and even, then
the Lie   algebra of $G$ does not have $\ss\ll (2,\CC)$ as a quotient.

\end{description}

Then every almost effective analytic action $(\alpha, G, M) $
 has a fixed point. 
\end{theorem}

\begin{proof}  
As $G$ contains a 1-dimensional normal subgroup, there exists
$X\zpe\G^\za$ spanning a 1-dimensional ideal.
Evidently  $\G^\za$ tracks $X$.   
The $X$-block $\msf i (X)$ is essential by Poincar\'e-Hopf,  and the
conclusion follows from Theorem \ref{th:main}.
\end{proof}

\begin{corollary}               \mylabel{coM3}

Assume: 
\begin{itemize}

\item $M$ is  complex and  compact, and   $\chi (M)\ne 0$. 

\item $G$ is  solvable.
\end{itemize}

Then every holomorphic action $(\alpha, G, M)$ has a fixed point.
\end{corollary}

\begin{proof}

Apply Theorem \ref{thM2}, taking into account that even if the action
is not almost effective, the Lie algebra $\G^\za$ is solvable.
\end{proof}

As remark \ref{reMnr} shows, the analog for real analytic surfaces
does not hold, but if $G$ is supersolvable it follows from Theorem
\ref{thM3} because $G$ since they always contain a normal subgroup of
dimension one. ({\sc Hirsch \& Weinstein} \cite{HW2000} proved this in
a different way, and later on {\sc Turiel} \cite{Turiel03} listed the
Lie groups, solvable or not, that act without fixed point when
$\chi(M)\znoi 0$.)

The existence of fixed points for continuous actions on compact real
surfaces with nonzero Euler characteristic was proved by {\sc Lima} \cite
{Lima64} for the group $\RR^n$, and {\sc Plante} \cite{Plante86} for
connected nilpotent Lie groups; result extended by {\sc Hirsch} \cite{
  Hirsch2014} to nilpotent local actions.  Moreover {\sc Lima} and
{\sc Plante} showed that that every compact surface supports a
continuous fixed-point free action by the affine group $Aff_+ (\RR)$,
the first non-nilpotent Lie group.  It belongs to the folklore that
this kind of action can be taken smooth (\cite{Belliart97}; for an 
elementary construction of a such action see \cite{Turiel16}). 

Related results are in the articles \cite{Belliart97, Bon-Sant2015,
 Hirsch2010, Hirsch2015, Schneider74, Stowe80, Turiel90}.

 
\section{Lie algebras on compact connected complex surfaces with no common zero}   \mylabel{seNU1}
 
 In this section $\FF=\CC$, $M$ is a compact connected complex surface and every object is
 holomorphic unless another thing is stated.
 Our purpose is to describe all the compact cases in which Theorem \ref{th:main} fails if 
 hypothesis  (*) is deleted. As we will see, essentially there are only two models.
 Before constructing them let us recall some well known things
 necessary later on. 
  
First consider a compact 1-codimensional  submanifold $P$ of a real or complex 
manifold $N$ and a vector field $X$ on $N$. Assume the existence of an open set 
$A\zcco P$ such that $Z(X)\zin A=P$.
\smallskip

Consider $p\zpe P$. Suppose that there exists an 1-dimensional
foliation $\mathcal F$ defined on an open set $p\zpe B\zco A$ such that:
\begin{itemize}
\item$X$ is tangent to (leaves of) $\mathcal F$.
\item $T_q N=T_q \mathcal F \zdi T_q P$ for every $q\zpe B\zin P$.
\end{itemize}

Notice that $\mathcal F$ is unique if it exists.

Let $L_p$ be the leaf of $p$. Then $X\zbv L_p$ inside this leaf has an isolated  singularity at $p$
and by definition, the index of this singularity is {\it the transversal index of $X$ at $p$}.

Clearly when the transversal index is defined at $p$, it is also defined around $p$ in $P$ and it is 
locally constant. Therefore if the transversal index of $X$ is defined at any point then it has to be
constant on each connected component of $P$, which allows us to refer to it as {\it the index of
$X$ at (any component of) $P$}. 
\smallskip

On the other hand recall that the sphere $S^2$  admits a unique complex structure
up diffeomorphism, usually represented by $\CC P^1$. Besides, the
group of  diffeomorphisms (biholomorphic maps) of $\CC P^1$  is the
projective group $PGL(2,\CC)$, which is the quotient of $SL(2,\CC)$ by
$\{\mbf{I},-\mbf{I}\}$, where $\mbf{I}\in SL(2,\CC)$ denotes the identity.
Thus $\mcal V^\om (\CC P^1 )$ equals the algebra of projective vector fields,
which is isomorphic to $\ss\ll(2, \CC)$.

Note that the structure group of holomorphic fibre bundles over $\CC P^1$
with fibre $\CC P^1$ is $PGL(2,\CC)$. Since $\ZZ_2$ is the fundamental
group of $PGL(2,\CC)$, from the real point of view, there only exist
two fibre bundles over $\CC P^1$ with fibre $\CC P^1$; more exactly
$\CC P^1 \zpor\CC P^1$ and $\CC P^2\sharp{\overline{\CC P}^2}$, which
is the result of blowing up a point of $\CC P^2$.

Actually in the $C^\zinf$-category there are only two fibre bundles
over $S^2$ with fibre $S^2$ (${\rm Diff_+ }(S^2 )$ strongly retracts
onto $SO(3)$, {\sc Smale} \cite{Smale59}).

  From the complex viewpoint things are different because a holomorphic
map from an open set $A\zco\CC$, which includes $S^1$, into
$SL(2,\CC)$ extends to $D^2$ in the $C^\zinf$-category but not always
like holomorphic map.

\begin{model}[${\msf Z}(X)$ is connected]  \mylabel{Mod1}

On $\CC P^1$ consider the projective
vector fields $\widetilde X$ which on
$\CC \zco\CC P^1$ is written as $z^2 {\frac {\zpar} {\zpar z}}$. 
Set $M=\CC
P^1 \zpor\CC P^1$; let $\zp_1 ,\zp_2$ be the canonical projections. On
the other hand let $\G\zco\mcal V^\om (M)$ be the Lie algebra tangent
to the first factor and isomorphic by $(\zp_1)_\ast$ to $\mcal V^\om (\CC P^1 )$,
and $X$ the vector field on $M$ tangent to the second factor whose
projection by $(\zp_2)_\ast$ equals $a\widetilde X$, $a\zpe
\CC\,\verb=\=\,\{0\}$. Clearly $[X,\G]=0$.

Moreover ${\msf Z}(\G)=\emp$ so ${\msf Z}(X)\zin{\msf Z}(\G)=\emp$,
while ${\msf Z}(X)$ is a 1-submanifold diffeomorphic to $\CC P^1$ and
transversally to it the index of $X$ equals 2.
\end{model}

\begin{remark}[A compactification construction]  \mylabel{reMaa}

Since $\CC^k \zco \CC P^k$ any linear transformation of $\CC^k$ is the restriction of
a projective transformation of $\CC P^k$ and $GL(k,\CC)$ can be regarded as a subgroup
of $PGL(k+1,\CC)$ in a natural way. 
Now consider a holomorphic line bundle $\zp\co E\zfl N$. Completing each fibre with its own infinity
point gives rise to a new holomorphic fibre bundle $\zp\co Q\zfl N$ with fibre $\CC P^1$ (for sake of 
simplicity the projection map is still denoted $\zp$). 

More exactly if $\{U_\zl \}_{\zl\zpe L}$ is a 
trivializing open covering of $N$ with transition functions 
$g_{\zl\zm}\co U_\zl \zin U_\zm \zfl GL(1,\CC)$ associated to $\zp\co E\zfl N$, then at the same time 
it is associated to $\zp\co Q\zfl N$ if $GL(1,\CC)$ is seen as a subgroup of $PGL(2,\CC)$ and,
therefore, every $g_{\zl\zm}$ takes its values in $PGL(2,\CC)$. 

Let $Q_0$ be (the image of) the zero section of $E$ and set 
$Q_\zinf \co =Q\,\verb=\=\,E$. Clearly
$Q_\zinf$ is a complex submanifold, that we will call the infinity section,
and $\zp\co Q_\zinf \zfl N$ a diffeomorphism.

The radial vector field $R$ of $E$ extends to a vector field on $Q$ still called $R$ since the 
diffeomorphism $z\zpe\CC\,\verb=\=\,\{0\}\zfl z^{-1}\zpe\CC\,\verb=\=\,\{0\}$ transforms
$z{\frac {\zpar} {\zpar z}}$ in $-z{\frac {\zpar} {\zpar z}}$. Besides
${\msf Z}(R)=Q_0 \zun Q_\zinf$ and transversally to $Q_0$ and $Q_\zinf$ the index of R
equals 1.  

Set $E'\co =Q-Q_0$. As any diffeomorphism $\zr\co\CC\zfl\CC$ which preserves
$-z{\frac {\zpar} {\zpar z}}$ is a linear automorphism, $\zp\co E'\zfl N$ has a natural
structure of holomorphic line bundle with zero section $Q_\zinf$ and radial vector field
$-R$. With respect to the open covering of $N$ giving before, its transition functions
$g'_{\zl\zm}$ are $g'_{\zl\zm}=g^{-1}_{\zl\zm}$. Thus $c_1 (E')=-c_1 (E)$ where $c_1$
denotes the first Chern class. Moreover if one adds the infinity point to each fibre of
$E'$ one obtains $\zp\co Q\zfl N$ again but this time $Q_0$ is the infinity section.

Of course the constructions above do not depend on the trivializing open covering of $N$.
\end{remark}

\begin{remark}  \mylabel{reMbb}

Let $\zp\co E\zfl N$ be a holomorphic vector bundle and $R$ its radial vector field. 
Consider a diffeomorphism $f\co E\zfl E$ that preserves $R$. Then $f$ maps fibres in
fibres, which induces a second diffeomorphism ${\tilde f} \co N\zfl N$ such that
${\tilde f}\zci\zp=\zp\zci f$ and every 
$f\co\zp^{-1}(q)\zfl\zp^{-1}({\tilde f}(q))$, $q\zpe N$, is a linear isomorphism.

Indeed, $f$ has to map ${\msf Z}(R)$, that is the zero section of $E$, into itself. On the
other hand each fibre $\zp^{-1}(q)$ is the set of all the points of $E$ that have the zero
of $\zp^{-1}(q)$ as $\za$-limit. 

Moreover if $\zp\co E\zfl N$ is a holomorphic line bundle and $\zp\co Q\zfl N$  its 
compactification given in Remark \ref{reMaa} then $f$ extends to a diffeomorphism
of $Q$ (obvious since each $f\co\zp^{-1}(q)\zfl\zp^{-1}({\tilde f}(q))$, is a linear
so projective). Thus any complete vector field $Y$ on $E$ such that $[Y,R]=0$ 
extends to a vector field on $Q$.

Indeed, let $\zF_t$ be the flow of $Y$. Then every $\zF_t$ preserves $R$, so maps
fibres in fibres, which implies that $Y$ is foliate with respect to fibres. Moreover $Y$
has to be tangent to ${\msf Z}(R)$, that is to the zero section. Therefore if $U$ is a
trivializing open set of $N$ and one identifies $\zp^{-1}(U)$ to $U\zpor\CC$ endowed 
with variables $(y,z)$ one has
$$Y(y,z)={\widetilde Y}(y)+\zf(y,z)\zpu z{\frac {\zpar} {\zpar z}}$$
where ${\widetilde Y}$ is a vector field on $U$. 

But $[Y,R]=0$ and $R=z{\frac {\zpar} {\zpar z}}$, so function $\zf$ only depends on 
$y$. Since the compactification of $U\zpor\CC$ is $U\zpor\CC P^1$ and clearly ${\widetilde Y}$
and $\zf\zpu z{\frac {\zpar} {\zpar z}}$ extend to $U\zpor\CC P^1$ so does $Y$.   
\end{remark}

\begin{model}[${\msf Z}(X)$ is not connected]  \mylabel{Mod2}

This kind of examples are constructed on the compactation given in Remark
\ref{reMaa} of holomorphic line bundles over $\CC P^1$. As the Picard group of
$\CC P^1$ is $\mathbb Z$ these line bundle are {\it holomorphically} classified by
their first Chern class \cite{GrifHarris78}. We will need:

\begin{lemma}  \mylabel{leMaa}

Let $\zp\co E\zfl \CC P^1$ be a holomorphic line bundle and $R$ its radial vector field.
Then there exists one and only one Lie algebra $\G\zco \V^\zw (E)$ such that:
 \begin{itemize}  
\item $[R,\G]=0$,    

\item $\G$ is isomorphic by $\zp_\zas$ to $ \V^\zw (\CC P^1)$.
\end{itemize}

Moreover $\G$ comes from an action of $SL(2,\CC)$ on $E$. Therefore its 
elements are complete vector fields.
\end{lemma}

\begin{proof}

{\underline{Uniqueness:}} let $\G,\H$ be as in the lemma. Since the elements of $\G$ and $\H$
are tangent  to the zero section and $\zp_\zas\co\G\zfl\V^\zw(\CC P^1 )$,
$\zp_\zas\co\H\zfl\V^\zw(\CC P^1 )$ isomorphisms, then every element of $\H$ writes as
$Y+a_Y R$ where $Y\zpe\G$ and $a_Y \co E\zfl\CC$ is holomorphic. Now
$[R,Y+a_Y R]=0$ implies that $a_Y$ is constant along fibres.Therefore 
$a_Y =b_Y \zci\zp$ where $b_Y \co\CC P^1 \zfl \CC$ is a holomorphic function, so constant.
In other words $\H=\{Y+a_Y R\co Y\zpe\G,\, a_Y \zpe\CC\}$. But $\mcal I
=\{Y\zpe\G\co a_Y =0\}$ is a nonzero ideal so $a_Y =0$ whatever
$Y\zpe\G$ and $\H=\G$.

{\underline{Existence:}} given a holomorphic line bundle $\zp\co E\zfl\CC P^1$ it suffices to show the 
existence of an action  $\za\co SL(2,\CC)\zpor E \zfl E$ which is almost effective and 
fibre preserving.  Recall that {\it fibre preserving} means the existence of a second action 
(the projected one on the basis $\CC P^1$) $\zb\co SL(2,\CC)\zpor\CC P^1 \zfl\CC P^1$ such that
$\zp(\za(g,e))=\zb(g,\zp(e))$ and
$\za(g,\quad)\co\zp^{-1}(p)\mapsto\zp^{-1}(\zb(g,p))$ is a linear isomorphism
for any $g\zpe SL(2,\CC)$, $e\zpe E$ and $p\zpe\CC P^1$.
For that one will show the existence of such action of $SL(2,\CC)$ for any value
of the first Chern class.

First observe that given two
almost effective actions $(\za,SL(2,\CC),E)$ and $(\za',SL(2,\CC),E')$
over actions $(\zb,SL(2,\CC),\CC P^1 )$ and $(\zb',SL(2,\CC),\CC
P^1 )$ respectively, if $\zb=\zb'$ then there exists a natural
almost effective action of $SL(2,\CC)$ on $E\zte E'$ over
$\zb=\zb'$. Besides $c_1 (E\zte E')=c_1 (E)+c_1 (E')$.

On the other hand an action of $SL(2,\CC)$ on $E$ induces an actions
on its dual vector bundle $E^\ast$, both over the same action on $\CC
P^1$ (recall that $c_1 (E^\ast )=-c_1 (E))$.

Thus it suffices to construct it on the canonical line bundle $E_1$ over $\CC P^1$.
But it is well known that the natural action of $SL(2,\CC)$ on $\CC^2$
induced a such action on $E_1$ first by setting $g\zpu(v,w)=(g\zpu v,g\zpu w)$ on
${\widetilde F}\co=\{(v,w)\zpe(\CC^2 \,\verb=\=\,\{0\})\zpor\CC^2 
\co v\zex w=0\}$,
and then by considering the induced action on the quotient $E_1$ of
$\widetilde F$ under the equivalence relation $(v,w)\mcal R (v',w')$
$\zsii$ $v\zex v'=0$ and $w=w'$.
\end{proof}

Consider a holomorphic line bundle $\zp\co E\zfl\CC P^1$ and its compactification
given in Remark \ref{reMaa}. Depending on the context denote by $\G$ the Lie
algebra of Lemma \ref{leMaa} or its extension to $M$, which  exists because
$\G$ on $E$  consists of complete vector fields (Remark \ref{reMbb}).
Analogously $R$ will be the radial vector field on $E$ or its extension to $M$.

{\it Always on $M$} note that $[\G,R]=0$ and $\G$ is isomorphic by 
$\zp_\zas$ to $\V^\zw (\CC P^1 )$, so ${\msf Z}(\G)=\zva$ since
${\msf Z}(\V^\zw (\CC P^1 ))=\zva$. Therefore $(M,X,\G)$ where
$X=aR$, $a\zpe\CC\,\verb=\=\,\{0\}$, is an example in which Theorem 
\ref{th:main} fails, and ${\msf Z}(X)=M_0 \zun M_\zinf$ where $M_0$ is
the zero section and $M_\zinf$ the infinity one ($Q_0$ and $Q_\zinf$ in 
the notation of Remark \ref{reMaa}). Obviously ${\msf Z}(X)$ possesses 
two connected components.

Clearly $\zp\co E\zfl\CC P^1$ and $\zp\co E'\zfl\CC P^1$ where $E'\co=M-M_0$ give
rise to the same example (up to a nonzero coefficient multiplying $X$), so these examples
only depend on the absolute value of the first Chern class.

When $\zbv c_1 (E)\zbv\znoi 0$ this number equals the order of the fundamental group
of $M-{\msf Z}(X)=M-(M_0 \zun M_\zinf)$. If $\zbv c_1 (E)\zbv=0$ then the fundamental 
group of  $M-{\msf Z}(X)$ is $\ZZ$. Thus $\zbv c_1 (E)\zbv$ is an invariant which classifies
this kind of examples up to a nonzero coefficient multiplying $X$.
\end{model}

In compact complex surfaces with nonvanishing Euler characteristic,
Models \ref{Mod1} and \ref{Mod2} are the only ways for constructing
Lie algebras tracking nontrivial vector fields with no common zero. More exactly:

\begin{theorem}         \mylabel{thM1A}

Let $M$ be a compact, connected, complex 2-manifold  with $\chi(M)\znoi
0$. Assume   $\G\subset \V^\om (M)$ is a finite-dimensional Lie algebra
that tracks a non-trivial vector field $X\zpe\V^\om (M)$. If 
$\Z X \zin\Z \G=\emp$ then the following conditions hold:
\begin{description}

\item[(a)] $M$ is a holomorphic fibre bundle
over $\CC P^1$ with fibre $\CC P^1$, hence
 $M$ is simply connected and $\chi (M)=4$.

\item[(b)] $\G$ contains a subalgebra $\mathcal A$ isomorphic to
$\ss\ll(2,\CC)$, with ${\msf Z}(\mcal A)=\emp$.

\item[(c)] $(M,X,\mathcal A)$ is holomorphically equivalent to the example of 
Model \ref{Mod1} or to an of the examples of Model \ref{Mod2}. 

\end{description}

\end{theorem}

This result will be proved in section \ref{seM5} (see Remark \ref{reMla} for more 
details on the algebra $\G$).

\begin{example}  \mylabel{exM7}

Let $M$ be a compact, connected, complex 2-manifold,  Assume:
\begin{itemize}

\item  $\chi(M)\znoi
0$.

\item
$\G\zco\mcal V^\zw(M)$ is  a subalgebra isomorphic to
${\mathfrak {g}}\ll(2,\CC)$.
 
\item $\mcal A\zco\G$ is the unique subalgebra of $\G$  isomorphic 
to $\ss\ll(2, \CC)$ \cite{Jacobson79}, and $X\in\G$ spans the center of $\G$.

\item ${\msf Z}(\G)\znoi\emp$. (The case ${\msf Z}(\G)=\emp$ is 
described by Theorem \ref{thM1A}.)
\end{itemize}

 Lemma \ref{leM2}, applied to $X$ and $\mcal A$, shows that ${\msf Z}(\G)$ 
 is a finite set, $\{p_1 ,\dots,p_r \}$, $r \ge 1$, and 
${\msf Z}(\G)={\msf Z}(\A)$.   
Blowing up all these points gives rise to another compact, connected
complex 2-manifold $M'$, a vector field $X'\in \V^\om (M')$ coming
from $X$, and a Lie algebra $\mcal A' \zco\mcal V^\zw(M' )$
isomorphic to $\mcal A$ (hence to $\ss\ll(2,\CC)$). Moreover 
${\msf Z}(\mcal A' )=\emp$;
indeed as $\A$ is simple $\mcal A_0 (p_k )/\mcal
A_1 (p_k )$ equals the special linear algebra $\ss\ll(T_{p_k}M)$
for each $k=1,\dots,r$ (see section \ref{seM4} for definitions) and this
zero of $\A$ is deleted by the blowup process \cite{GrifHarris78}.

It follows from  the blowup construction that  $M$ and $M'$ have the same
fundamental group and $\chi(M')=\chi (M) +r$. Theorem \ref{thM1A}
shows that $M'$ is simply connected and $\chi(M')=4$. Therefore $M$ is also 
simply connected and $\chi(M)\zmei 3$. Since $S^4$ has no complex structure,
topologically $M$ is $\CC P^2$ and $r=1$.

Notice that the linear part of $X$ at $p_1$ has to be 
$a\mbf{Id}$, $a\zpe \CC \,\verb=\=\,\{0\}$, otherwise as $\mcal A_0 (p_1 )/\mcal
A_1 (p_1 )=\ss\ll(T_{p_1}M)$ all $k$-jets of $X$ at $p_1$ vanish and
$X=0$. Therefore transversally to $\CC P^1 \zeq M'
\,\verb=\=\,(M\,\verb=\=\,\{p_1 \})$ the index of $X'$ equals 1 and
$(M' ,X' ,\mcal A' )$ follows Model \ref{Mod2} for the canonical
line bundle $E_1$ since the normal vector bundle of
$ M' \,\verb=\=\,(M\,\verb=\=\,\{p_1 \})$ is isomorphic to $E_1$. 

As the examples of Model \ref{Mod2} are determined by the absolute value
of their first Chern class if one considers a second manifold  $N$,
$\H \zco\V^\zw (N)$, $\B\zco\H$ and $Y\zpe\H$ in the same
conditions as $M$, $\G$, $\A$ and $X$, and one blows up the only singular point
$q_{1}$ then there is a (holomorphic) diffeomorphism $\zf\co M'\zfl N'$ which
transforms  $\A'$ in $\B'$ and $X'$ in $aY'$ for some $a\zpe\CC \,\verb=\=\,\{0\}$.
Besides
$\zf(  {M' \,\verb=\=\,(M\,\verb=\=\,\{p_1 \})}  )
={N' \,\verb=\=\,(N\,\verb=\=\,\{q_1 \})}   $
because the first Chern class of their respective normal vector bundles are the same, 
that of $E_1$ (notice that the first Chern class of the other component of ${\msf Z}(X')$, 
or of ${\msf Z}(Y')$, is the opposite one).  

Now crushing ${M' \,\verb=\=\,(M\,\verb=\=\,\{p_1 \})}$ and
${N' \,\verb=\=\,(N\,\verb=\=\,\{q_1 \})}$ respectively into a point gives rise to a 
homeomorphism $\zq\co M\zfl N$ with $\zq(p_1 )=q_1$. Clearly
$\zq\co  {M\,\verb=\=\,\{p_1 \}}\zfl {N\,\verb=\=\,\{q_1 \}}$ is biholomorphic, so
$\zq\co M\zfl N$ is biholomorphic too and transform $\G$ in
$\H$. 

In other words, {\it up to isomorphism there is only one example of manifold $M$ and 
subalgebra $\G$ as above}. For instance: $M=\CC P^2$ and $\G$ the subalgebra
of those projective vector fields that on $\CC^2 \zco\CC P^2$ are linear.

Let $G$ be a connected Lie group and let $N,P$ be two compact connected manifolds
(real or complex it does not matter). Given two actions $\za\co G\zpor N\zfl N$ and
$\zb\co G\zpor P\zfl P$ we will say they are {\it equivalent} if there exist a diffeomorphism
$\zq\co N\zfl P$ and an automorphism $\zl\co G\zfl G$ such that
$$\za(g,p)=\zq^{-1}(\zb(\zl(g),\zq(p)))$$
for any $(g,p)\zpe G\zpor N$.
Assume $\za$ and $\zb$ are effective. From \cite{Palais57} follows that if there is a
diffeomorphism $\zq\co N\zfl P$ which transforms $Im\, d_\za$ in $Im\, d_\zb$ then 
$\za$ and $\zb$ are equivalent.

Therefore {\it any effective holomorphic action of $GL(2,\CC)$ on a connected compact
complex surface $M$, with $\chi(M)\znoi 0$, which possesses a fixed point is
equivalent to the natural action of $GL(2,\CC)$ on $\CC P^2$}.

\end{example} 


\section{Other examples}   \mylabel{seM2}

In this part we give several examples of finite dimensional Lie
algebras on surfaces tracking nontrivial vector fields, focusing
on complex surfaces, the most difficult case. 

\begin{example}   \mylabel{exM1}

Recall that in dimension two a projective vector field $Y$ means a
fundamental vector field of the natural action of $SL(3,\FF)$ on $\FF
P^2$, which is effective if $\FF=\RR$ and almost effective when
$\FF=\CC$ (its kernel equals $\{a\mbf{I} \co (a\zpe\CC,\, a^3 =1)\}$,
where $\mbf{I}\in SL(3,\CC)$ denotes the identity).  
The restriction of $Y$ to $\FF^2 \zco\FF
P^2$ can be written as
\[a_1 {\frac {\zpar} {\zpar x_{1}}}+a_2 {\frac {\zpar} {\zpar
    x_{2}}} +\zsu_{k,r=1,2}b_{kr}x_{k} {\frac {\zpar} {\zpar x_{r}}}
+\left(c_1 x_1 +c_2 x_2 \right)\left(x_1 {\frac {\zpar} {\zpar x_{1}}}
+x_2 {\frac {\zpar} {\zpar x_{2}}}\right)\]
with $a_1, a_2, b_{k r},c_k \zpe\FF$.

Now on $\FF^2$ consider the Lie algebra $\G$ corresponding to the
projective vector fields on $\FF P^2$ that vanish at the origin.  This
algebra tracks $X=x_1 {\frac {\zpar} {\zpar x_{1}}}+x_2 {\frac {\zpar}
{\zpar x_{2}}}$, itself belonging to $\G$ since $[Y,X]=-(c_1 x_1+c_2
x_2 )X$. Note that $\G$ has no ideal of dimension one, but the notion
of tracking will allow us to bridge this gap. Note also that $X$ and
$\G$ extend to $\FF P^2$, but the tracking functions do not.

Observe that  ${\msf Z}(X)={\msf Z}(\G)=\{(0,0)\}$, and this is an
essential $X$-block.

From $\FF P^2$ we construct a new 2-manifold $M'$ over $\FF$ by
blowing up the origin in $\FF^2$, and on it a vector field $X'$ and a
Lie algebra $\G'\subset \V^\omega (M')$, isomorphic to $\G$, which
tracks $X'$ (see \cite{GrifHarris78}).  Note that the blowup of the
origin is ${\msf Z}(X')$, an $X'$-block diffeomorphic to $\FF P^1$.  
As the linear part of $X$ at the origin is the identity, the transversal index
of $X'$ at ${\msf Z}(X')$  equals one hence
${\msf i}_{{\msf Z}(X')}(X')=\chi({{\msf Z}(X')})=\chi(\FF P^1 )$.
Thus in the real case this block is inessential (for $X'$) while it is
essential with index 2 in the complex one.  Moreover
$\ss\ll (2,\FF)$ is a quotient of $\G'$ because clearly it is a quotient of $\G$. 

From the blowup construction follows $\G'$ acts transitively on ${\msf Z}(X')$
since any linear vector field on $\FF$ belongs to $\G$, so $\Z {\G'}=\emp$.  
Therefore:
\begin{itemize} \item {\em  For the complex case  of
Theorem \ref{th:main},  the supplementary hypothesis {\em (*)} 
cannot be deleted even in the non-compact case.}

\end{itemize}

If we consider the solvable subalgebra $\G'_0$ of $\G'$,
corresponding to $\G_0 \zco\G$ defined by the supplementary condition
$b_{21}=0$, then ${\msf Z}(X')\zin {\msf Z}(\G'_0)\znoi\emp$ since
$\G'_0$ vanishes at the point of $\FF P^1$ associated to the second
axis. In turn, blowing up this common zero gives rise to a new
manifold endowed with a Lie algebra $\G''_0$, isomorphic to $\G_0$ and
$\G'_0$, and a vector field $X''$ tracked by $\G''_0$. Now ${\msf
  Z}(X'')$ is again essential, more exactly ${\msf i}_{{\msf
    Z}(X'')}(X'')$ equals $-1$ in the real case and 3 in the complex
one.  Therefore ${\msf Z}(\G''_0)\zin {\msf Z}(X'')\znoi\emp$.

For easily computing the index of $X$, $X'$ and $X''$, notice that as
a real vector field $X$ is outwardly transverse to the spheres $S^1
\zco \RR^2$ or $S^3\zco \CC^2$. Therefore in each case this index
equals the Euler characteristic of the ambient manifold.
\end{example}

\begin{example}   \mylabel{exM1A}

In a more general setting, let $\mcal P_n$ be the vector space of
homogeneous polynomials in $x_1 ,x_2$ of degree $n\zmai 1$ over $\FF$ and let
$\G$ be the $(n+5)$-dimensional Lie algebra of vector fields on
$\FF^2$ of the form
$$\zsu_{k,r=1,2}b_{kr}x_{k} {\frac {\zpar} {\zpar x_{r}}}
+\zf\zpu\left(x_1 {\frac {\zpar} {\zpar x_{1}}} +x_2 {\frac {\zpar}
  {\zpar x_{2}}}\right)$$
where $b_{k r}\zpe\FF$ and $\zf\zpe\mcal P_n$.

As in Example \ref{exM1}, set $X=x_1 {\frac {\zpar} {\zpar x_{1}}}+x_2
{\frac {\zpar} {\zpar x_{2}}}$. Then $\G$ tracks $X$.

Blowing up the origin in $\FF^2$ provides a new 2-manifold $M$ over
$\FF$ endowed with a Lie algebra $\G'\subset \V^\omega (M)$ and
a vector field $X'$ which is tracked by $\G'$. As before ${\msf
  Z}(X')=\FF P^1$ and ${\msf Z}(\G')=\emp$ so ${\msf Z}(\G')\zin {\msf
  Z}(X')=\emp$, while ${\msf i}_{{\msf Z}(X')}(X')$ equals zero if
$\FF=\RR$ and two when $\FF=\CC$.

Notice that the dimension of $\G'$ can be taken as large as
desired. Thus:
\begin{itemize} 
\item {\it In the noncompact complex case, when Theorem \ref{th:main}
fails the respective Lie algebra has $\ss\ll(2,\CC)$ as a quotient but
its dimension can be arbitrarily large} 
\end{itemize} 
(see Remark \ref{reMla} for the compact
case).

Of course one may consider the subalgebra $\G_0 \zco\G$ given by
the condition $b_{21}=0$ and do as in Example \ref{exM1}.
\end{example}

\begin{example}   \mylabel{exM2}

In $\FF^3$ with coordinates $x=(x_1 ,x_2 ,x_3 )$ let $S_{\FF}^2$ be
the ``sphere'' given by the equation $x_{1}^2 +x_{2}^2 +x_{3}^2 =1$.
This is the real sphere $S^2$ if $\FF=\RR$.  When $\FF=\CC$ it is a
noncompact complex 2-manifold whose underlying real
manifold is the tangent vector bundle   $TS^2$;  but $S_{\CC}^2$ is not
biholomorphic to $T\CC P^1$ because  $\CC P^1$ is 
never a complex submanifold of $\CC^3$.  On $S_{\FF}^2\zco\FF^3$
consider the tangent vector fields
$$X= {\frac {\zpar} {\zpar x_{3}}}-x_{3}\left(x_1 {\frac {\zpar}
{\zpar x_{1}}} +x_2 {\frac {\zpar} {\zpar x_{2}}}+x_3 {\frac {\zpar}
  {\zpar x_{3}}}\right),\quad Y= -x_2 {\frac {\zpar} {\zpar
    x_{1}}}+x_1 {\frac {\zpar} {\zpar x_{2}}}.$$

Denote by $\mcal P_k$ the space of homogeneous polynomials in $x_1
,x_2$ of degree $k$. Set 
\[
  \G_k =\{aY+\zf X\co a\zpe\FF,\zf\zpe\mcal P_k\},
\] 
which is a $(k+2)$-dimensional solvable Lie algebra that
tracks $X$.

On $S_{\FF}^2$ our $X$ has just two singular points $(0,0,\zmm 1)$,
each of them an essential block of index 1. Indeed, first observe that
functions $x_1 ,x_2$ can be regarded as coordinates of $S_{\FF}^2$
around $(0,0,\zmm 1)$ which we name $(u_1 ,u_1 )$.  As $X\zpu x_k
=-x_3 x_k$, $k=1,2$, up to sign the linear part of $X$ at $(0,0,\zmm
1)$ equals $u_1 {\frac {\zpar} {\zpar u_{1}}}+u_2 {\frac {\zpar}
 {\zpar u_{2}}}$.  Note that $\G_k$ vanishes at these points. Blowing
up $(0,0,1)$ and $(0,0,-1)$ gives rise to a 2-manifold $M$, a vector
field $X'$ with two isolated blocks $K_1 ,K_{2}$ associated to these
points, and a Lie algebra $\G'$ that is isomorphic to $\G$ and tracks
$X'$.
\begin{itemize} 
\item{\it But the behavior of the real and complex cases is quite different.}
\end{itemize}

Indeed, in the complex one $K_1$ and $K_{2}$ are diffeomorphic to $\CC P^1$ 
and are thus  essential blocks and $\G'$ vanishes somewhere in $K_1$ and in
$K_{2}$. In the real case $M$ is the Klein bottle, $K_1 ,K_2$ are
$S^1$ so non-essential blocks, and $\G'$ does not vanishes at any
point of $M$.
\end{example}

\begin{example}  \mylabel{exM3}

Let $\G$ be the Lie algebra on $\CC P^2$ of those projective vector
fields $Y$ that on $\CC^2 \zco \CC P^2$ write
$$Y=a_1 {\frac {\zpar} {\zpar z_{1}}}+a_2 {\frac {\zpar} {\zpar
    z_{2}}} +\zsu_{k,r=1,2}b_{kr}z_{k} {\frac {\zpar} {\zpar z_{r}}}$$
where $a_k ,b_{k r}\zpe\CC$ and $b_{21}=0$.  Then $\G$ is a 3-solvable
Lie algebra of dimension five and the vector field represented by
${\frac {\zpar} {\zpar z_{2}}}$ spans an ideal of dimension one.
Moreover $\G$ vanishes at the infinity point of the second axis
(belonging to $\CC P^2 -\CC^2$).

Now by blowing up this point one constructs a second complex surface
$M_1$, of Euler characteristic 4 and simply connected, endowed with a
Lie algebra $\G_1$ isomorphic to $\G$. Again there is some zero of
$\G_1$, which can be blown up to construct a simply connected manifold
$M_2$ of Euler characteristic 5, and a Lie algebra $\G_2\subset \V^\om
(M_2)$ isomorphic to $\G$ and so on. Therefore:

\begin{itemize}
\item{\it For any $m\zmai 3$ there exists a simply connected compact
complex 2-manifold of characteristic $m$ that supports a 3-solvable Lie
algebra of vector fields.}
\end{itemize}

A similar process can be started from the product of two copies of
$\CC P^1$, endowed each of them with a 2-dimensional noncommutative
Lie algebra, and the product algebra.

\end{example}

\begin{example}   \mylabel{exM4}

Here we show that for every integer $m$ and every positive integer $d$
there is a compact complex 2-manifold  $M$ and a solvable Lie algebra
$\A\subset \V^\om (M)$ such that 
\[
\chi (M) =m, \quad \dim_\CC\, \A =d,  \quad \dim_\CC\, \Z\A =1.
\]

 First consider a real or complex $n$-manifold\, $N$, with cotangent 
bundle $\zp\co T^\zas N\zfl N$.  On the manifold $T^* N$  define
the {\it Liouville 1-form} $\zr$ by 
\[
 \zr(v)=\za(\zp_{*}v), \qquad \big(\za\in T^* N, \ v \in T_{\za}(T^* N)\big)
\]
The {\it  
  Liouville symplectic form} on $T^*N$ is the exterior 
2-form  $\zw=d\zr$.  Given a 1-form $\zb$ on the manifold
$T^* N$, assumed analytic over the ground field, define the vector
field
\begin{equation}                \label{eq:x}
X_\zb\in \V^\om (T^*N), \qquad \iota_{X_\zb}\om =\zb
\end{equation}
%
where $\iota_{X_\zb}\om$ is the interior product  of  $\om$
by $X_\beta$.
 
Now let $N$ be a  1-dimensional complex manifold defined by a
orientable compact connected surface of genus 
$g\zmai 1$ endowed with a K\"ahler structure. 
The Dolbeault cohomology group of $N$ in dimension 1, which is
isomorphic to the singular cohomomlogy group $H^1 (N,\RR)$, has a
basis represented by $g$ holomorphic 1-forms 
\[\alpha_j:=\zl_j +i\zm_j.
\]

Using Equation (\ref{eq:x}), set 
\[
 X_j:=X_{\zp^* \za_j} \in \V^\om (T^*N),\qquad (j=1,\dots,g)\]
 and let $X_{g+1}$ denote the
radial vector field on $T^* N$.
By means of coordinates it is easily checked that $X_1 ,\dots,X_g$
and $X_{g+1}$ are tangent to the fibres $T_{p}^* N$,
($p\zpe N$),   with each $X_j$, $j=1,\dots,g$, constant and 
$X_{g+1}$ linear.   Moreover $[X_j ,X_{g+1}]=X_j$, $j=1,\dots,g$.

Let $\zp\co Q\zfl N$ be the compactification of $\zp\co T^\zas N\zfl N$ given
by Remark \ref{reMaa}. The vector fields $X_k$, $k=1,\dots,g+1$,  extend to holomorphic 
vector fields $\hat X_k \in \V^\om (Q)$ such that $Q_\zinf \subset \Z{\hat X_k}$ where 
$Q_\zinf$ denotes the infinity section of $Q$.    
It is easy to see that $\hat X_1 ,\dots,\hat X_{g+1}$ form a basis of a solvable
complex Lie algebra $\G\subset \V^\om (Q)$ of dimension $g+1$.  Evidently
$\Z\G$ is the union of the 1-dimensional complex submanifold $Q_\zinf$ and the 
image by the zero section of the set  of common zeros of $\za_1 ,\dots,\za_g$.
Because the holomorphy this last set is always finite, so we can reasonably  
write $\dim_\CC\, \Z\G =1$.

Note that  $\chi (Q)=4(1-g)$.  By 
blowing up  $r$ zeros of $\G$ we obtain a compact complex 2-manifold
$M$ with $\chi (M)=4(1-g)+r$ and a solvable Lie algebra $\G'\subset \V^\om (M)$
isomorphic to $\G$. Finally, take $g$ and $r$ such that $4(1-g)+r=m$ and 
$g\zmai d$, and a Lie subalgebra $\A\subset\G'$ of dimension $d$. 

\end{example}

\section{Consequences of tracking} \mylabel{seM3}

Throughout this section we assume:

\begin{itemize}

\item {\em $P$ is a real or complex $n$-manifold with empty boundary.}

 \item {\em $X, Y$ are differentiable vector fields on $P$ and $Y$
   tracks $X$ with tracking function $f$.}

\end{itemize}

When $P$ is complex and $X$ nontrivial on each connected component of
$P$, our function $f$ is holomorphic. Indeed, locally in coordinates
$f$ is meromorphic because $[Y,X]$ ``divided'' by $X$ equals $f$, so
holomorphic since it is continuous.  {\it Thus if $P$ is compact,
  connected and complex, the tracking function $f$ is constant.}

\bigskip

In the real case $f$ can be just continuous at some point. For
instance, on $\RR$ set $Y=x^4 {\frac {\zpar} {\zpar x}}$ and $X=g
{\frac {\zpar} {\zpar x}}$ where $g(x)=e^{-1/x}$ if $x>0$,
$g(x)=e^{-1/x^2 }$ if $x<0$ and $g(0)=0$. A computation shows that
$f(x)=x^2 -4x^3$ if $x>0$, $f(x)=2x-4x^3$ if $x<0$ and $f(0)=0$; hence
it is not derivable at the origin.
\bigskip

If $Y_1$ is another differentiable vector field tracking $X$ with
function $f_1$, and $f,f_1$ are at least $C^1$, which automatically
holds if $\FF=\CC$, from the Jacobi identity follows $[Y,Y_1 ]$ tracks
$X$.
\bigskip

By definition the {\em dependency set} of $X$ and $Y$ (over the ground
field $\FF$) is
\[
  \msf D (X,Y)=\big\{p\in M\co (X\wedge_\FF Y)(p)=0\big\}.
\]
\begin{proposition} \mylabel{prM1} 
If $Y$ tracks $X$ then ${\msf Z}(X)$ and $\msf D (X, Y)$ are $X$- and
$Y$-invariant.
\end{proposition}

\begin{proof}  Evidently ${\msf Z}(X)$ is $X$-invariant. Let us see
  its $Y$-invariance. Consider an integral curve $\zg\co A\zfl P$ of
  $Y$ where $A$ is a connected open set $\FF$.  Suppose $\zg(t_0 )\zpe
  {\msf Z}(X)$; then $\zg(t)\zpe{\msf Z}(X)$ for any $t\zpe A$
  sufficiently close to $t_0$. Indeed, if $Y(\zg(t_0 ))=0$ it is
  obvious; otherwise as the statement is local by means of suitable
  coordinates we may assume $P$ is a product of intervals ($\FF=\RR$)
  or a polydisk ($\FF=\CC$) always centered at $\zg(t_0
  )=(0,\dots,0)$, $Y= {\frac {\zpar} {\zpar x_{1}}}$ and
  $X=\zsu_{k=1}^n g_k {\frac {\zpar} {\zpar x_{k}}}$. Now $[Y,X]=fX$
  means

\begin{equation}                \label{eqM1}
{\frac {\zpar g_k} {\zpar x_{1}}}=fg_{k},\quad k=1,\dots,n.
 \end{equation}

 Since $f$ is continuous ($\FF=\RR$) or holomorphic ($\FF=\CC$) the
 general solution to equation (\ref{eqM1})
 $$g_k (x)=h_k (x_2 ,\dots,x_n )e^\zf,\quad k=1,\dots,n,$$
 where $ {\frac {\zpar\zf} {\zpar x_{1}}}=f$ and $\zf(x)=0$ whenever
 $x_1 =0$.  But $\zg(t_0 )=0$ so each $h_k (0,\dots,0)=0$ and $X$
 vanishes along the first axis.

Therefore the set $A'=\{t\zpe A \co \zg(t)\zpe{\msf Z}(X)\}$ is open
and closed, so $A'=A$ or $A'=\emp$, which proves the $Y$-invariance of
${\msf Z}(X)$.

The $X$- and $Y$-invariance of $\msf D (X, Y)$ is proved in the same
way by taking into account that $L_X (X\wedge_\FF Y)=0$ and letting  $L_Y
(X\wedge_\FF Y)=f X\wedge_\FF Y$.
\end{proof}

\section{Vanishing of the index and other results}   \mylabel{seM4}

In the first lemma of this section we assume:
{\it\begin{itemize} 
\item $P$ is a real $n$-manifold
  
\item $X\in \V^{\zinf} (P)$

\item $K$ is an $X$-block and $U\subset P$ is isolating for $(X, K)$.
\end{itemize}}

\begin{lemma}           \mylabel{leM1}

Assume:
\begin{description}

\item[(a)] $K$ is a compact submanifold,

\item[(b)] $X_1, \dots,X_r\in \V^{\zinf} (U)$ are tangent to $K$, and
  their values are linearly independent at each point of $K$,

\item[(c)] $ \displaystyle X = \sum_{j=1}^rf_jX_j$ on $U$.
 
\end{description} 
Let $\mcal D\subset \V (M)$ be a neighborhood of $X$ in the
compact-open topology.  Then there exists $X'\in \mcal D$ such that:
\[X' = X \ \,\text{on} \,M\,\verb=\=\,U, \qquad 
{\msf Z}(X') \cap U =\emp.
\]
  Then ${\msf i}_K (X)=0$.
\end{lemma}

\begin{proof}

Take a tubular neighborhood $ W$ of $K$ such that $W \subset U$ and
identify it with an orthogonal vector bundle $\pi\co E \to K$, with
the norm in each fibre denoted by $\|\cdot\|$.  Set
$ E_a := \{e \in E \co \|e\| < a\}$ for each $a >0$,
so that $ \ov E_a \subset U $.
 Let $\varphi_a\co M \to \Rp$ be a non-negative function with support
 in $E_a $ such that $\varphi_a^{-1} (0)\cap K = \emp$ . If $ a $ is
 small enough, for each $ e \in E_a $ the vector subspace spanned by $
 X_1 (e),\dots, X_r (e) $ is ”almost transverse” to the fibre of $ e
 $, so its intersection with $T_e (\pi^{−1} (π(e))) $ equals $ \{0\}$.  
 Let $R$ denote the radial vector field on $E$.

For sufficiently small $ \eps > 0 $, the vector field $X'\in \V (M)$
defined as
\[ X'|\,U:= X + \eps\varphi_a {(X_{1}+R)},   \qquad   X'|\, M\verb=\=W :X|\, M\verb=\=W 
\]
has the required properties.
\end{proof}


Let us recall some well known facts on jets of vector fields useful later on. 
Consider a set $\mathcal B$ of vector fields on a manifold $Q$.
Given $p\zpe Q$ and $k\geq 0$  set                    
${\mathcal B}_{k}(p)\co=\{Y\zpe\mathcal B\co j^{k}_{p}Y=0\}$, while
${\mathcal B}_{-1}(p)\co=\mathcal B$ and
${\mathcal B}(p)\co=\{Y(p) \co Y\zpe\mathcal B\}\zco T_p Q$. Every 
${\mathcal B}_{k-1}(p)/{\mathcal B}_{k}(p)$, $k\geq 0$, can be regarded as a
subset of $T_{p}Q\zte S^{k}(T_{p}^{*}Q)$ and ${\mathcal
B}_{-1}(p)/{\mathcal B}_{k}(p)$ as a subset of the set of all
polynomial vector fields on $T_{p}Q$ of degree $\leq k$. When
${\mathcal B}$ is a Lie algebra $[{\mathcal B}_{k}(p), {\mathcal
B}_{r}(p)]\zco {\mathcal B}_{k+r}(p)$, $k+r\geq -1$, therefore
each ${\mathcal B}_{k}(p)$, $k\geq 0$, is a Lie algebra and every
${\mathcal B}_{k+s}(p)$, $k,s\geq 0$, an ideal of 
${\mathcal B}_{k}(p)$. If $Q$ is connected and ${\mathcal B}$ a finite
dimensional analytic Lie algebra then ${\mathcal B}_{k}(p)$, $k\geq 1$, 
is nilpotent and ${\mathcal B}_{r}(p)=0$ for some $r$.

A (piecewise differentiable) {\it curve tangent to $\mathcal B$} is a
finite family $\mcal C$ of integral curves $\zg_j \co A_j \zfl Q$,
$j=1,\dots,k$, of elements of $\mcal B$ defined on connected open
subsets of $\FF$, such that $A_r \zin A_{r+1}\znoi\emp$,
$r=1,\dots,k-1$.  We say that $\mcal C$ {\it joins} $p$ to $q$ if
$p\zpe A_1$ and $q\zpe A_k$.  The {\it $\mcal B$-orbit} of $p$ is the
set of all points $q\zpe Q$ joined to $p$ by some curve tangent to
$\mcal B$. 

When ${\mathcal B}(p)= T_p Q$ it is easily seen that $p$ belongs to
the interior of its ${\mathcal B}$-orbit. 
If $\mcal B$ is a finite dimensional Lie algebra then the
$\mcal B$-orbit of $p$ is a submanifold of $P$, not always regular,
whose dimension equals that of the vector subspace ${\mathcal
  B}(p)\zco T_p Q$, \cite{Palais57} (a submanifold of $Q$ is called
{\it embedded or regular} when its topology as submanifold and that as
topological subspace of $Q$ coincide).

From Lie's work \cite{Lie24} follows :

\begin{lemma}           \mylabel{leNew1}

Let $Q$ be an 1-dimensional connected manifold and let $\mcal B\zco\V^{\zw}(Q)$
be a nonzero finite-dimensional Lie algebra. Then $\mcal B$ is isomorphic to
either $\FF$, the affine algebra of $\FF$, or $\ss\ll(2,\FF)$.  
\end{lemma}

In the next three results assume:
{\it \begin{itemize} 
\item $N$ is a connected complex surface.
  
\item $\A\zco \V^{\zw} (N)$ is a Lie algebra isomorphic to $\ss\ll
  (2,\CC)$ and $X\in \V^{\zw} (N)$ is nontrivial.
\end{itemize}}

\begin{lemma}           \mylabel{leM2}

The points of ${\msf Z}(\A)$ are isolated, and ${\msf Z}(\A)\zco
{\msf Z}(X)$ when $[X,\mcal A]=0$.
\end{lemma}

Indeed, if $p\zpe{\msf Z}(\mcal A)$ then since $\A$ is simple
$\A_{0}(p)/\A_{1}(p)$ equals the special linear algebra $\ss\ll(T_p
N)$, so $p$ is isolated in ${\msf Z}(\mcal A)$. Besides, if $[X,\A]=0$
being isolated clearly implies $X(p)=0$.

\begin{lemma}           \mylabel{leM3}

Assume $[X,\mcal A]=0$. Consider a point $p\zpe N$ such that
$X(p)\znoi 0$. Then around $p$ there exist coordinates $z=(z_1 ,z_2
)$, with $p\zeq 0$, such that the vector fields
$$Y_1 = {\frac {\zpar} {\zpar z_{1}}},\quad Y_2 =z_{1}{\frac {\zpar}
{\zpar z_{1}}} +a{\frac {\zpar} {\zpar z_{2}}}, \quad Y_3
=z^{2}_{1}{\frac {\zpar} {\zpar z_{1}}} +2az_{1}{\frac {\zpar} {\zpar
    z_{2}}}\, ,\,\, a\zpe\CC,$$
are a basis of the restriction of $\A$ to the domain of coordinates,
and $X={\frac {\zpar} {\zpar z_{2}}}$.
\end{lemma}

\begin{proof}

As $\A$ is simple its projection on the local quotient $N'$ of $N$ by
(the foliation associated to) $X$ is either zero or a Lie algebra
$\A'$ isomorphic to $\A$. If zero, each $Y\zpe\A$ is proportional to
$X$, which is incompatible with the hypothesis $[X,\A]=0$. Therefore
as $\dim\, N'=1$ there exists a coordinate $z_1$, around the projection
$p'$ of $p$ in $N'$, such that $z_1 (p')=0$ and
$${\frac {\zpar} {\zpar z_{1}}},\quad z_{1}{\frac {\zpar} {\zpar
    z_{1}}}, \quad z^{2}_{1}{\frac {\zpar} {\zpar z_{1}}}$$
span $\A'$.

This coordinate $z_1$ may be regarded as a function around $p$ and
adding a new function $z_2$, such that $z_2 (p)=0$ and $X\zpu z_2 =1$,
leads to a system of coordinates $z=(z_1 ,z_2 )$, defined on a domain
identified through $z$ to a polydisk centered at the origin (shrink it
if necessary), in which $p\zeq 0$, $X={\frac {\zpar} {\zpar z_{2}}}$
and
$$Y_1 = {\frac {\zpar} {\zpar z_{1}}}+f_{1}{\frac {\zpar} {\zpar
    z_{2}}}, \quad Y_2 =z_{1}{\frac {\zpar} {\zpar z_{1}}}
+
f_{2}{\frac {\zpar} {\zpar z_{2}}}, \quad Y_3 =z^{2}_{1}{\frac
  {\zpar} {\zpar z_{1}}} +f_{3}{\frac {\zpar} {\zpar z_{2}}}\, ,$$
for suitable functions $f_1 ,f_2 ,f_3$, span $\A$ (more exactly its
restriction to the domain of coordinates). Observe that each $f_k$
only depends on $z_1$ because $[X,Y_k ]=0$.

Taking $z_2 +g(z_1 )$ instead of $z_2$ for a suitable function $g(z_1
)$ allows us to suppose $Y_1 = {\frac {\zpar} {\zpar z_{1}}}$. Then as
$[Y_1 ,Y_2 ]=Y_1$, $[Y_1 ,Y_3 ]=2Y_2$ and $[Y_2 ,Y_3 ]=Y_3$ (project
into $\A'$ to  see this),  a straightforward computation shows that
$f_2 ,f_3$ are as stated.
\end{proof}

\begin{lemma}           \mylabel{leM4}

Consider a point $p\zpe N$ such that $X(p)=0$ and $dim\A(p)=1$.
\begin{description}

\item[(a)]  If $\mcal A$ tracks $X$ then either all orbits of $\mcal
A$ near $p$ have dimension one and $X$ is tangent to them, or about
$p$ there exist coordinates $z=(z_1 ,z_2 )$, with $p\zeq 0$, an integer
$n\zmai 1$ and functions $h(z_1 , z_{2})$, $f(z_2 )$, $g(z_2 )$ with
$h=(0,0)\znoi 0$ and $f(0)=g(0)=0$ such that
$$X=h(z_1 , z_{2}) z^{n}_{2}{\frac {\zpar} {\zpar z_{2}}}$$ and the
vector fields
$$Y_1 = {\frac {\zpar} {\zpar z_{1}}},\quad Y_2 =z_{1}{\frac {\zpar}
{\zpar z_{1}}} +f(z_2 ){\frac {\zpar} {\zpar z_{2}}},\quad Y_3
=z^{2}_{1}{\frac {\zpar} {\zpar z_{1}}} +\left( 2z_{1}f(z_2 )+g(z_2 )
\right){\frac {\zpar} {\zpar z_{2}}}$$
are a basis of $\A$ (under restriction).

\item[(b)] If $[X,\mcal A]=0$ then about $p$ there exist
coordinates $z=(z_1 ,z_2 )$, with $p\zeq 0$, an integer $n\zmai 1$ and
scalars $a\zpe\CC\,\verb=\=\,\{0\}$, $ b\zpe\CC$ such that
$$X=a z^{n}_{2}{\frac {\zpar} {\zpar z_{2}}}$$
and the vector fields
$$Y_1 = {\frac {\zpar} {\zpar z_{1}}},\quad Y_2 =z_{1}{\frac {\zpar}
{\zpar z_{1}}} +b z^{n}_{2}{\frac {\zpar} {\zpar z_{2}}}, \quad Y_3
=z^{2}_{1}{\frac {\zpar} {\zpar z_{1}}} +2bz_{1} z^{n}_{2}{\frac
  {\zpar} {\zpar z_{2}}}$$
are a basis of $\A$ (under restriction).
\end{description}

\end{lemma}

\begin{proof}

The $\A$-orbit of $p$ is a submanifold $N'$ of dimension one and $\A$
is tangent to it. Therefore there exist a coordinate $u$, defined on a
small open set $p\zpe N''\zco N'$ with $u(p)=0$, and three vector
fields $Y_1 ,Y_2 ,Y_3$ which span $\A$ such that under restriction to
$N''$ respectively write
$$Y_1 = {\frac {\zpar} {\zpar u}},\quad Y_2 =u{\frac {\zpar} {\zpar
    u}}, \quad Y_3 =u^{2}{\frac {\zpar} {\zpar u}}\, .$$

Thus $[Y_1 ,Y_2 ]=Y_1$, $[Y_1 ,Y_3 ]=2Y_2$ and $[Y_2 ,Y_3 ]=Y_3$.

Around $p$ in $N$ our $u$ can be extended to a function $z_1$ such
that $Y_1 \zpu z_1=1$. Take a second function $z_2$ vanishing at $p$
such that $Y_1 \zpu z_2 =0$ and $(dz_1 \zex dz_2 )(p)\znoi 0$; then
$z=(z_1 ,z_2 )$ near $p\zeq 0$ is a system of coordinates with domain
of polydisk type (shrink it if necessary) such that $Y_1 ={\frac
  {\zpar} {\zpar z_{1}}}$. Note that $z_2 =0$ defines an open set of
$N'$ which includes $p$. Besides as $[{\frac {\zpar} {\zpar
      z_{1}}},Y_2 ]=[Y_1 ,Y_2 ]=Y_1 ={\frac {\zpar} {\zpar z_{1}}}$,
$[{\frac {\zpar} {\zpar z_{1}}},Y_3 ]=[Y_1 ,Y_3 ]=2Y_2$ and $Y_2
(p)=Y_3 (p)=0$ one has
$$Y_2 =\left( z_{1}+f_1 (z_2 )\right){\frac {\zpar} {\zpar z_{1}}}
+
f_2 (z_2 ){\frac {\zpar} {\zpar z_{2}}},$$
$$Y_3 =\left( z^{2}_{1}+2z_1 f_1 (z_2 )+g_1 (z_2 ) \right) {\frac
  {\zpar} {\zpar z_{1}}} +\left( 2z_{1}f_2 (z_2 )+g_2 (z_2 )
\right){\frac {\zpar} {\zpar z_{2}}}$$
where $f_1 (0)=f_2 (0)=g_1 (0)=g_2 (0)=0$.

In case (b), $[Y_1 ,X]=0$. In case (a) since $Y_1$ tracks $X$ and $Y_1
(p)\znoi 0$ there always exists a function $\zl$, defined around $p$,
with no zero such that $[Y_1 ,\zl X]=0$. Therefore as $\mcal A$ tracks
$\zl X$ too and it suffices to prove the result for $\zl X$, we may
assume $[Y_1 ,X]=0$ without lost of generality. Thus $X=h_1 (z_2
){\frac {\zpar} {\zpar z_{1}}}+h_2 (z_2 ){\frac {\zpar} {\zpar
 z_{2}}}$ with $h_1 (0)=h_2 (0)=0$. Since $X$ is nontrivial, there
exists an integer $n\zmai 1$ such that $X=z^{n}_{2}(\zf_1 (z_2 ){\frac
 {\zpar} {\zpar z_{1}}}+\zf_2 (z_2 ){\frac {\zpar} {\zpar z_{2}}})$
where at least $\zf_1 (0)\znoi 0$ or $\zf_2 (0)\znoi 0$.

Assume $\zf_1 (0)\znoi 0$ and $\zf_2 (0)=0$. Set $\widetilde X=\zf_1
(z_2 ){\frac {\zpar} {\zpar z_{1}}}+\zf_2 (z_2 ){\frac {\zpar} {\zpar
    z_{2}}}$; then $\widetilde X \zex[\widetilde X ,\mcal A]=0$
because $\mcal A$ tracks $X$. Let $Q$ be the quotient
1-manifold, around $p$, of $N$ by $\widetilde X$ and let $q$ be the
projection of $p$.  Then $\mcal A$ projects onto a Lie algebra $\mcal
B$ of vector fields on $Q$, which is either zero or simple. But
$\mcal B(q)=0$ since $(\widetilde X \zex Y_1 )(p)=0$ and $Y_2 (p)=Y_3
(p)=0$, so $\mcal B =0$. In other words $\mcal A$ is tangent to
$\widetilde X$ and $Y_1 \zex\widetilde X =0$ everywhere, which proves
the first possibility in the case (a).  In the case (b) one has $X=h_1
(z_2 ){\frac {\zpar} {\zpar z_{1}}}$ and $Y_2 =\left( z_{1}+f_1 (z_2
)\right){\frac {\zpar} {\zpar z_{1}}}$ so $[X,Y_2 ]\znoi 0$, {\it contradiction}.

In short we may suppose $\zf_2 (0)\znoi 0$. Now taking $z_1 +g(z_2)$
instead $z_1$ for a suitable function $g(z_2)$ allows us to suppose
${\widetilde X}=\zf_2 (z_2 ){\frac {\zpar} {\zpar z_{2}}}$ and
$X=z_{2}^{n}\zf_2 (z_2 ){\frac {\zpar} {\zpar z_{2}}}$. On the other
hand it is well known that $z_2$ can be modified in such a way that
$X=az_{2}^{n}{\frac {\zpar} {\zpar z_{2}}}$, $a\zpe\CC\,\verb=\=\,\{0\}$.

The remainder of the proof easily follows from the fact that $\mcal A$
tracks $X$ (case (a)) or $[X,\A]=0$ (case (b)) and $[Y_2 ,Y_3 ]=Y_3$,
and it is left to the reader.
\end{proof}

\section{Proofs of the main theorems}     \mylabel{seM5}

\subsection {Proof of Theorem \ref{th:main}}  
Take an isolating open set $U$ that we can suppose connected since at
least the index of $X$ at one of its components has to be nonzero in
the general case and negative or odd (or both) in the particular one
(by particular case we mean that $M$ is complex and it has
$\ss\ll(2,\CC)$ as a quotient). 
As there always exists a decreasing sequence $\{M_r \}_{r\zpe\mathbb N}$ of
precompact connected open sets of $M$ any of them containing $\overline U$ 
such that $\zing_{r\zpe\mathbb N}{\overline M}_r =\overline U$, replacing
$M$ with a suitable $M_r$ allows to suppose $U$ ``almost-equal'' $M$
and ${\msf Z}(X)=K$.

We  assume henceforth
$${\msf Z}(\G)\zin K=\emp,$$ otherwise the proof is finished. 
Then $K$ is a regular submanifold of dimension one.

Indeed, clearly $\dim\,\G(p)\zmai 1$ for any $p\zpe K$. 
Now if $\dim\,\G (q) =2$ for some $q\zpe K$, as by
Proposition \ref{prM1} the $q$-orbit under $\G$ is included in $K$,
then $\Int K\znoi \emp$ and, by analyticity, $K=U$ that is $X=0$, {\em
  contradiction}.  In short, $\dim\,\G(p)=1$, $p\zpe K$. Given any
$p\zpe K$, take $Y\zpe\G$ with $Y(p)\znoi 0$ and consider coordinates
$(A,x_1 ,x_2 )$ around $p\zeq (0,0)$ as in the proof of Proposition
\ref{prM1}; let $T$ be the transversal to $Y$ defined by $x_1
=0$. Then $K\zin T$ is the set of zeros of two analytic functions one
of them at least nontrivial; so it is isolated in $T$ and only one
point if $A$ is sufficiently small. In this last case $A\zin K$ is given by
the equation $x_2 =0$, which shows that $K$ is a regular submanifold
of dimension one.

One may assume $K$ connected and still ${\msf Z}(X)=K$ by shrinking
$U$ and $M$, if necessary, since the index of $X$ at some of the
components has to be nonzero in the general case and negative or odd
in the particular one.

First suppose the existence of $Y\zpe\G$ such that ${\msf Z}(Y)\zin
K=\emp$. If the dependency set $\msf D (X,Y)$ equals $M$ then on a
small open neighborhood of $K$ apply Lemma \ref{leM1} to $X,X_1 =Y$ if
$\FF=\RR$ or to $X,X_1 =Y,X_2 =iY$ regarded as real vector fields if
$\FF=\CC$, for concluding ${\msf i}_K (X)=0$, {\em contradiction}.

Therefore $\msf D (X,Y)\znoi M$. Since $\msf D (X,Y)\zin
(M\,\verb=\=\,{\msf Z}(Y))$ is $Y$-saturated reasoning as in the case
of $K$ shows that $\msf D (X,Y)\zin (M\sm{\msf Z}(Y))$ is a regular
1-submanifold, which clearly included $K$. Therefore $K$ is a
component of $\msf D (X,Y)\zin (M\sm{\msf Z}(Y))$ and there exists an
open neighborhood $A$ of $K$ such that $A\zin\msf D (X,Y)=K$.  Let
$\phi: M \to \RR$ be a function with a very narrow support around $K$
such that $\phi(K) = 1$; regard $X$ and $Y$ as real vector
fields. Then the vector field $X_1 := X + \eps\phi Y$, $\eps >0$, has
no zero on $M$ and ${\msf i}_K (X)=0$, {\em contradiction}.

In short, ${\msf Z}(Y)\zin K\znoi\emp$ for every $Y\zpe\G$.

Thus $K=S^1$ in the real case (obvious) and $K=\CC P^1$ in the complex
one. Indeed, there always exists $Y\zpe\G$ whose restriction to $K$
does not vanish identically; but this restriction has zeros, all of
them of positive index because the holomorphy, so $\chi (K)>0$.

Define
\[\I:= \{Y \in \G\co Y|K = 0\},
\]
which is an ideal in $\G$.  The image of $\G$ in $\V^\om (K)$ maps
$\G/\I$ isomorphically onto a subalgebra $\H$ of $\V^\om (K)$.
Moreover each element of $\H$ vanishes somewhere and $\H$ is
transitive, that is $\dim\,\H(p)=1$ for any $p\zpe K$, because ${\msf
Z}(\G)\zin K=\emp$. 

As $\dim\, K=1$, from Lemma \ref{leNew1} follows that 
up to isomorphism $\H$ has to be either $\{0\}$,  a 1-dimensional
algebra (both of them clearly non-transitive), the affine algebra of
$\FF$, or $\ss\ll(2,\FF)$.

First assume $\FF=\CC$ and $\H=\ss\ll(2,\CC)$, so we are in the
particular case and ${\msf i}_K (X)$ is negative or odd. On the other
hand since $\ss\ll(2,\CC)$ is simple there exists a subalgebra
$\mcal A$ of $\G$ isomorphic under restriction to $\H$ \cite{Jacobson79}.

Consider a point $p\zpe K$. If near $p$ all orbits of $\mcal A$ have
dimension one and $X$ is tangent to them, by analyticity since $K$ is
an $\mcal A$-orbit of dimension one and $X$ is tangent to it there is
an open neighborhood $D$ of $K$ such that any $\mcal A$-orbit on $D$
has dimension one and $X$ is tangent to it. Thus on $D$ the Lie
algebra $\mcal A$ defines a 1-foliation $\mcal F$ to which $X$ is
tangent. But $K$ is a compact simply connected leaf of $\mcal F$, so
near it the foliation $\mcal F$ is a product.

Let $L\znoi K$ be a compact leaf sufficiently close to $K$. Then $X$,
which is tangent to $L$, does not vanish on it so $\chi(L)=0$. But
topologically $L$ is $\CC P^1$, {\it contradiction}.

Therefore from (a) of Lemma \ref{leM4} applied to $X$ and $\mcal A$
(we have just seen that the first alternative of (a) is forbidden)
follows that ${\msf i}_K (X)$ equals $2n>0$ since transversally to $K$
the index of $X$ is $n$. But we are in the particular case and ${\msf
  i}_K (X)$ is negative or odd, {\it contradiction again}. In short
never $\H=\ss\ll(2,\CC)$.

Now assume $\FF =\RR$ and $\H = \ss\ll(2, \RR )$.  Then there is
$T\zpe\H$ which does not belong to any 2-subalgebra; for instance if
$ad_T$ has some non-real eigenvalue \cite{Jacobson79}. This means that $T$
never vanishes on $K$, otherwise if $T(q)=0$ for some $q\zpe K$ then
$T$ belongs to the 2-dimensional subalgebra $\H_{0}(q)$. Therefore
$\ss\ll(2,\RR)$ is excluded as well.

Finally if $\H$ is the affine algebra of $\FF$ there exists a basis
$\{T_1 ,T_2 \}$ of $\H$ such that $[T_1 ,T_2 ]=T_2$. But $T_2 (q)=0$
for some $q\zpe K$ so $T_1 (q)=0$ too, otherwise $[T_1 ,T_2 ]\znoi
T_2$. In other words the affine algebra is not transitive.

Summing up there is no way for choosing the subalgebra $\H$, so
assuming ${\msf Z}(\G)\zin K=\emp$ leads to contradiction.
{\it Therefore the proof is finishes}.

\subsection {Proof of Theorem \ref{thM1A}}

First recall some elementary facts on $\ss\ll(2,\CC)$:

\begin{remark} \mylabel{reMsl}

Given $\zf\zpe\ss\ll(2,\CC)\,\verb=\=\,\{0\}$ one can find a basis $\{e_1 ,e_2
\}$ of $\CC^2$ such that:
\begin{itemize}

\item $\zf=a(e_1 \zte e_{1}^\ast -e_2 \zte e_{2}^\ast )$,
  $a\zpe\CC\,\verb=\=\,\{0\}$, if $\zf$ is invertible. In this case the connected
  subgroup of $SL(2,\CC)$ whose subalgebra is spanned by $\zf$ is
  closed and isomorphic to the multiplicative group $\CC\,\verb=\=\,\{0\}$.
\item $\zf=e_2 \zte e_{1}^\ast$ if $\zf$ is not invertible. Now the
  connected subgroup determined by $\zf$ is closed and isomorphic to
  $\CC$.
\end{itemize}

Therefore the projective vector field $Y_\zf$ associated to $\zf$ can
be identified, under conjugation by $PGL(2,\CC)$, to that whose
restriction to $\CC\zco\CC P^1$ is written as:
\begin{itemize}

\item $2az{\frac {\zpar} {\zpar z}}$, $a\zpe\CC\,\verb=\=\,\{0\}$, if $\zf$ is
  invertible. Then the vector field $Y_\zf$ possesses two
  singularities on $\CC P^1$ both of them of index 1 and eigenvalues
  of its linear part $\zmm2(-det\zf)^{1/2}$.
\item ${\frac {\zpar} {\zpar z}}$ if $\zf$ is not invertible. Observe
  that ${\frac {\zpar} {\zpar z}}$ and $z^2 {\frac {\zpar} {\zpar z}}$
  are $PGL(2,\CC)$ conjugated, so $Y_\zf$ can be represented by $z^2
  {\frac {\zpar} {\zpar z}}$ too.
 \end{itemize}

\end{remark}

{\bf And now the proof of Theorem \ref{thM1A}.}

First recall that if $Q$ is a connected complex manifold of dimension
one and $\B\zco V^\om (Q)$ a finite dimensional Lie algebra such that
${\msf Z}(\B)\znoi\emp$, then $\B$ is solvable (see Lemma \ref{leNew1}).

Now assume ${\msf Z}(X)\zin {\msf Z}(\G)=\emp$. Reasoning as in the
proof of Theorem \ref{th:main} shows that ${\msf Z}(X)$ is a compact
1-submanifold of $M$; obviously non-empty since $\chi(M)\znoi
0$. Moreover it has to exist a component $K$ of ${\msf Z}(X)$
diffeomorphic to $\CC P^1$, such that the image $\H$ of $\G$ in $\V^\om
(K)$ under the restriction is isomorphic to $\ss\ll (2,\CC)$;
otherwise $K\zin{\msf Z}(\G)\znoi\emp$.

As $\H$ is simple there is a subalgebra $\A$ of $\G$ isomorphic
under restriction to $\H$. This algebra $\A$ tracks $X$, so $[Y,X]=a_Y
X$, $a_Y \zpe \CC$, for any $Y\zpe\A$.  But $\{Y\zpe\A \co a_Y =0\}$
is a nonzero ideal of $\A$; therefore every $a_Y =0$ and
$[X,\A]=0$. Let $P$ be a component of ${\msf Z}(X)$. If $P\zin{\msf
  Z}(\A)\znoi\emp$ then the restriction of $\A$ to $P$ has to be
solvable, so zero. That is ${\msf Z}(\A)\zcco P$, which contradicts
Lemma \ref{leM2}. In short ${\msf Z}(\A)\zin{\msf Z}(X)=\emp$, so
${\msf Z}(\A)=\emp$ since again by Lemma \ref{leM2} ${\msf
  Z}(\A)\zco{\msf Z}(X)$, and $\A$ acts transitively on every
component of ${\msf Z}(X)$. All these components are spheres that is
$\CC P^1$ because on each of them some $Y\zpe\A \,\verb=\=\,\{0\}$ has a zero.

On the connected open set $M\sm{\msf Z}(X)$ the vector field $X$ defines
a complex 1-dimensional foliation $\mathcal F$, which is the real
2-foliation associated to the commuting vector fields $X,iX$. Thus its
leaves are planes, cylinders or tori because $X$ is complete on
$M\sm{\msf Z}(X)$.

Each leaf of $L$ of $\mathcal F$ is closed in $M\sm{\msf Z}(X)$. Let us
see it; given $p\zpe L$, Lemma \ref{leM3} shows the existence of
$Y\zpe\A\,\verb=\=\,\{0\}$ and an open set $p\zpe D\zco M\sm{\msf Z}(X)$ such that
the connected component of $L\zin D$ relative to $p$ is included in
${\msf Z}(Y)\zin D$.  Therefore $L\zco{\msf Z}(Y)\zin(M\sm{\msf Z}(X))$
since $[X,Y]=0$ implies that ${\msf Z}(Y)$ is $X$-invariant; even more
${\msf Z}(Y)\zin (M\sm{\msf Z}(X))$ is a union of leaves of $\mcal F$.

But the same lemma shows that ${\msf Z}(Y)\zin(M\sm{\msf Z}(X))$ is a
closed regular 1-submanifold of $M\sm{\msf Z}(X)$. Since different
leaves of $\mathcal F$ are disjoint it follows that $L$ is a component
of ${\msf Z}(Y)\zin(M\sm{\msf Z}(X))$; in other words $L$ is an open and
closed subset of ${\msf Z}(Y)\zin(M\sm{\msf Z}(X))$, so closed in
$M\sm{\msf Z}(X)$.

On the other hand since $M$ is compact, the group $SL(2,\CC)$ acts on
$M$, with infinitesimal action $\A$, and on $M\sm{\msf Z}(X)$ as
well. Observe that this action is $\mathcal F$-foliate and
transversally transitive. Therefore given $L_1 ,L_2 \zpe\mathcal F$
there always exists $g\zpe SL(2, \CC)$ such that $g\zpu L_1 =L_2$.

Now take $p\zpe{\msf Z}(X)$ and consider coordinates $z=(z_1 ,z_2 )$
like in (b) of Lemma \ref{leM4} with domain $A$ of polydisk type. Then
the trace of $\mathcal F$ on $A$ is given by the slices $z_1
=constant$, and ${\msf Z}(X)\zin A$ by $z_2 =0$. Thus the set defined
by the conditions $z_1 =0$ and $z\znoi (0,0)$ is included in a leaf
$L$ of $\mathcal F$. Therefore $L$ and any leaf of $\mathcal F$ are
noncompact, that is they are (real) planes or cylinders. Moreover
$L\zun\{p\}$ as real surface has one end less than $L$. In this way
adding the points of ${\msf Z}(X)$ to the leaves of $\mathcal F$ gives
rise to a new complex 1-foliation $\mathcal F'$ on $M$, whose trace on
$M\sm{\msf Z}(X)$ is $\mathcal F$, and the action of $SL(2,\CC)$ on the
set of its leaves is still transitive.

Notice that any leaf $\tilde L$ of $\mathcal F$ at most intersects two
slices of $A$; indeed, if $S$ is a slice of A and $S\zin\tilde
L\znoi\emp$ then this non-empty intersection defines an end of $\tilde
L$. Therefore since the leaves of $\mathcal F$ are closed in $M\sm{\msf
  Z}(X)$ those of $\mathcal F'$ are closed, so compact, in $M$.

The procedure above kills the ends of every leaf of $\mathcal F$, so
each leaf of $\mathcal F'$ is topologically the sphere $S^2$ so $\CC
P^1$. Since the action of $SL(2,\CC)$ is transversally transitive the
foliation $\mathcal F'$ is given by a (complex) fibre bundle $\zp\co
M\zfl Q$ where $Q$ is a compact connected complex 1-manifold.

Finally from $[X,\A]=0$ follows that $\A$ projects onto a Lie algebra
$\A'\zco\mcal V^\om (Q)$ which is isomorphic to $\ss\ll
(2,\CC)$. Therefore $Q$ is the projective line so from now on we will
write $\zp\co M\zfl\CC P^1$. Observe that $\chi(M)=4$.

Let $P$ be a component of ${\msf Z}(X)$; as we said before $P$ is the
complex projective line. Consider a leaf $L'$ of $\mathcal F'$; then
$P\zin L'$ is a singleton. Indeed, some leaf $L''$ of $\mathcal F'$
intersects $P$ so any does because they are interchangeable under
$SL(2,\CC)$. As every point of $P\zin L'$ means an end of $L'\sm{\msf
  Z}(X)$, the set $P\zin L'$ does not have more than two
elements. Suppose it has two; then on $P$ we define the equivalence
relation $p_1 \mathcal R p_2$ if and only if there exists
$L''\zpe\mathcal F'$ such that $P\zin L''=\{p_1 ,p_2 \}$. It is easily
checked that the quotient $P/\mathcal R$ possesses a structure of
complex 1-manifold but, topologically, is $\RR P^2$ {\it
contradiction}.

From Lemma \ref{leM4} follows that transversally to $P$ the index of
$X$ equals $n$. Since $\chi(M)=4$ we have just two possibilities:
\begin{itemize}

\item ${\msf Z}(X)$ has two components and $n=1$, that is the leaves
  of $\mathcal F$ are cylinders.
\item ${\msf Z}(X)$ is connected and $n=2$, that is the leaves of
  $\mathcal F$ are (real) planes.
\end{itemize}

{\it First assume the leaves of $\mathcal F$ are cylinders.}

Set ${\msf Z}(X)=P_1 \zun P_2$ as union of its components. The
eigenvalue of the linear part of $X$ transversally to $P_1$ is a
holomorphic function on $P_1$, so constant. Thus considering $aX$
instead of $X$ for a suitable $a\zpe\CC\,\verb=\=\,\{0\}$, allows assuming that
this eigenvalue equals one.  Therefore for each $\zp^{-1}(q)$,
$q\zpe\CC P^1$, the projective vector field $X$ has two singularities
of index 1 and the eigenvalue of its linear part at $\zp^{-1}(q)\zin
P_1$ equals 1.  By Remark \ref{reMsl} $(\zp^{-1}(q)\zin (M-P_2 ),X)$
is diffeomorphic to $(\CC,z{\frac {\zpar} {\zpar z}})$.

Since any diffeomeomorphism $\zr\co\CC\zfl\CC$ which preserves
$z{\frac {\zpar} {\zpar z}}$ is a linear automorphism and the action
of $SL(2,\CC)$ associated to $\mathcal A$ preserves $X$, it follows
that $\zp\co M-P_2 \zfl\CC P^1$ is a line fibre bundle endowed with a
fibre action of $SL(2,\CC)$. Now it is obvious that $X$ and $\mathcal
A$ follow the construction of Model \ref{Mod2}.
\bigskip

{\it Now assume the leaves of $\mathcal F$ are planes.}

Then $\zp\co M\sm{\msf Z}(X)\zfl\CC P^1$ is a homotopy equivalence. As
$[X,\mathcal A]=0$ and $X$ is transversal to the orbits of the action
of $SL(2,\CC)$ on $M\sm{\msf Z}(X)$, they are diffeomorphic so with the
same dimension.

If this dimension equals two then $M\sm{\msf Z}(X)$ is an orbit of the
action of $SL(2,\CC)$; indeed, orbits in $M\sm{\msf Z}(X)$ are open and
$M\sm{\msf Z}(X)$ is connected. Take $p\zpe M\sm{\msf Z}(X)$; by Lemma
\ref{leM3} any $Y\zpe\mathcal A\,\verb=\=\,\{0\}$ such that $Y(p)=0$ is the
fundamental vector field associated to some
$\zf\zpe\ss\ll(2,\CC)\,\verb=\=\,\{0\}$ with $det\zf=0$ (consider the first
variable).  Therefore $M\sm{\msf Z}(X)$ is, as homogeneous space, the
quotient of $SL(2,\CC)$ by a closed subgroup $H$ (the isotropy group
of $p$) whose identity component is isomorphic to $\CC$.  Now the
homotopy sequence of the fibre bundle
$$H\zfl SL(2,\CC)\zfl M\sm{\msf Z}(X)$$ shows that $\zp_2 (M\sm{\msf
  Z}(X))=0$ (topologically $SL(2,\CC)$ is $S^3 \zpor \RR^3$). But
$\zp_2 (M\sm{\msf Z}(X))=\zp_2 (\CC P^1 )=\ZZ$ {\it contradiction}.

Thus the action of $SL(2,\CC)$ on $M$ defines a second foliation of
dimension one $\mathcal F''$ transversal to $\mathcal F'$. Observe
that ${\msf Z}(X)$ is a compact leaf of $\mathcal F''$ diffeomorphic
to $\CC P^1$, so simply connected. Therefore near ${\msf Z}(X)$ the
foliation $\mathcal F''$ is a product.

As $X$ is transversal to $\mathcal F''$ on $M\sm{\msf Z}(X)$, then all
the leaves of $\mathcal F''$ are $\CC P^1$ and $\mathcal F''$ is
defined by fibre bundle $\zp'\co M\zfl Q'$ where $Q'$ is a compact
connected complex 1-manifold. But $X$ projects onto $Q'$ in a
nontrivial vector field, so $Q'$ is $\CC P^1$ and the fibration
becomes $\zp'\co M\zfl\CC P^1$.

In short $\zp\zpor\zp'\co M\zfl\CC P^1 \zpor \CC P^1$ is a local
diffeomorphism so a covering and, finally, a diffeomorphism because
$\CC P^1 \zpor \CC P^1$ is simply connected. Thus $\zp\zpor\zp'$
identifies $M$ and $\CC P^1 \zpor \CC P^1$ in such a way that
$\mathcal F'$ is the foliation associated to the second factor and
$\mathcal F''$ that given by the first factor. Now it is obvious that
$X$ and $\mathcal A$ are constructed like in Model \ref{Mod1}. This
finishes the proof of Theorem \ref{thM1A}.

\begin{remark}[On the algebra $\G$ of Theorem \ref{thM1A}] \mylabel{reMla}

Consider $\G$ as in Theorem \ref{thM1A} and the fibre bundle $\zp\co
M\zfl\CC P^1$ given in this result. Let $\mcal A_m \zco\mcal V^\om
(M)$ be the maximal Lie algebra which includes $\mcal A$ and tracks
$X$ (actually it normalizes $X$ because $M$ is compact). Clearly $\G$
is a subalgebra of $\mcal A_m$, and  $\mcal A_m$  projects onto the Lie
algebra of projective vector fields of $\CC P^1$. Set
$$\mcal I_m \co=\{Y\zpe\mcal A_m \co \zp_\ast(Y)=0\}$$
that is an ideal of $\mcal A_m$.

For every $q\zpe\CC P^1$, $\mcal I_m \zbv\zp^{-1}(q)$ is included in
the normalizer of $X \zbv\zp^{-1}(q)$, and as $\zp^{-1}(q)=\CC P^1$
follows that $\mcal I_m \zbv\zp^{-1}(q)$ is a subalgebra of the
Lie algebra of projective vector fields, with dimension one if ${\msf
Z}(X)$ consists of two connected components and two if ${\msf Z}(X)$
is connected (see the proof of Theorem \ref{thM1A}).

In the first case every $Y\zpe\mcal I_m$ is written as $Y=fX$ where $f$ is a
holomorphic function so constant. That is to say $\mcal I_m$ equals
$\CC\{X\}\co=\{aX\co a\zpe\CC \}$, which implies that $\mcal A_m$ is
the direct product of $\mcal I_m$ and $\mcal A$. Therefore $\G$ is
$\mcal A$ or $\mcal A_m$.

Now assume ${\msf Z}(X)$ is connected. Then $M=\CC P^1 \zpor \CC P^1$,
$\zp$ is the first projection, $\mcal A$ can be seen as the Lie
algebra of projective vector fields on the first factor and $X$ as a
vector field on the second factor (see the proof of Theorem
\ref{thM1A} again). Hence one may choose a vector field $\widehat X$
tangent to the second factor such that $[\widehat X,X]=X$, $[\widehat
  X,\mcal A]=0$ and $\{X\zbv\zp^{-1}(q),\widehat X\zbv\zp^{-1}(q)\}$
is a basis of the normalizer of $X\zbv\zp^{-1}(q)$, $q\zpe\CC P^1$.

Therefore if $Y\zpe\mcal I_m$ then $Y=aX+\hat a \widehat X$, $a,\hat a
\zpe\CC$ (coefficients have to be holomorphic functions so
constant). Thus $\mcal A_m$ is the direct product of $\mcal A$ and the
2-dimensional Lie algebra spanned by $X,\widehat X$, while $\G$ equals
either $\mcal A_m$, $\mcal A$ or the direct product of $\CC\{X\}$ and
$\mcal A$.

Summing up, in compact connected complex surfaces Theorem \ref{th:main}
only fails with three Lie algebras: $\ss\ll(2,\CC)$, ${\mathfrak
  {g}}{\mathfrak {l}}(2,\CC)$ and the product of $\ss\ll(2,\CC)$ with
the affine algebra of $\CC$ (compare this fact to Example
\ref{exM1A}).
\end{remark}


\end{document}